\documentclass[a4paper]{article}
\usepackage[utf8]{inputenc}
\usepackage[T1]{fontenc}
\usepackage{amsmath,amssymb,amsfonts,amsthm,mathrsfs,comment,xcolor,color,bm}
\usepackage{tikz,amsmath,mathtools}
\usetikzlibrary{calc}
\usetikzlibrary{arrows}
\usepackage{verbatim}

\usepackage{array}
\usepackage[margin=0.8in]{geometry}
\usepackage{enumitem, hyperref}
\makeatletter
\def\namedlabel#1#2{\begingroup
    #2%
    \def\@currentlabel{#2}%
    \phantomsection\label{#1}\endgroup
}

\newcommand*\colvec[3][]{
    \begin{bmatrix}\ifx\relax#1\relax\else#1\\\fi#2\\#3\end{bmatrix}
}

\newcommand{\R}{\mathbb{R}}

\newcommand{\RD}{\mathbb{R}^d}
\newcommand{\N}{\mathbb{N}}
\newcommand{\Z}{\mathbb{Z}}

\newcommand{\cE}{\mathcal{E}}

\newcommand{\cH}{\mathcal{H}}

\newcommand{\cL}{\mathcal{L}}

\newcommand{\cP}{\mathcal{P}}

\newcommand{\cS}{\mathcal{S}}

\newcommand{\bes}{\begin{equation*}}
\newcommand{\ees}{\end{equation*}}
\newcommand{\beas}{\begin{eqnarray*}}
\newcommand{\eeas}{\end{eqnarray*}}
\newcommand{\bea}{\begin{eqnarray}}
\newcommand{\eea}{\end{eqnarray}}
\newcommand{\be}{\begin{equation}}
\newcommand{\ee}{\end{equation}}
\newcommand{\bei}{\begin{itemize}}
\newcommand{\eei}{\end{itemize}}
\newcommand{\bec}{\begin{cases}}
\newcommand{\eec}{\end{cases}}
\newcommand{\ben}{\begin{enumerate}}
\newcommand{\een}{\end{enumerate}}

\newcommand{\bbP}{\mathbb{P}}
\newcommand{\bbE}{\mathbb{E}}

\newcommand{\bbl}{\begin{block}}
\newcommand{\ebl}{\end{block}}

\newcommand{\De}{\mathrm{d}}

\newcommand{\rmD}{\mathrm{D}}

\newcommand{\rmP}{\mathrm{P}}

\newtheorem{mydef}{Definition}[section]

\newtheorem{prop}{Proposition}[section]
\newtheorem{theorem}{Theorem}[section]
\newtheorem{lemma}{Lemma}[section]

\newtheorem{remark}{Remark}[section]

\newtheorem{cor}{Corollary}[section]

\def\namedlabel#1#2{\begingroup
    #2%
    \def\@currentlabel{#2}%
    \phantomsection\label{#1}\endgroup
}

\newcommand{\inv}{\mathbf{m}}

\newcommand{\breta}{\bar{\eta}}

\newcommand{\cop}{\bm{\cL}^{\mathrm{cpl}}} 
\newcommand{\couprate}{\bm{c}^{\mathrm{cpl}}}
\newcommand{\bgamma}{\bar{\gamma}}
\newcommand{\zero}{\mathrm{e}}
\newcommand{\coupgrad}{\bm{\nabla}}
\newcommand{\supp}{\mathcal{S}}

\newcommand{\fphi}{f^{\phi}}
\newcommand{\poiss}{\mathrm{\bm \lambda}}

\title{A probabilistic approach to convex $(\phi)$-entropy decay for Markov chains}

\author{Giovanni Conforti\thanks{D\'epartement de Math\'ematiques Appliqu\'ees, \'Ecole Polytechnique, Route de Saclay, 91128, Palaiseau Cedex, France. giovanni.conforti@polytechnique.edu }}

\begin{document}

\maketitle 
We study the exponential dissipation of entropic functionals along the semigroup generated by a continuous time Markov chain and the associated convex Sobolev inequalities, including MLSI and Beckner inequalities. We propose a method that combines the Bakry \'Emery approach and coupling arguments, which we use as a probabilistic alternative to the discrete Bochner identities. The method is well suited to work in a non perturbative setting and we obtain new estimates for interacting random walks beyond the high temperature/weak interaction regime. In this framework, we show that the exponential contraction of the Wasserstein distance implies MLSI. We also revisit classical examples often obtaining new inequalities and sometimes improving on the best known constants. In particular, we analyse the zero range dynamics, hardcore and Bernoulli-Laplace models and the Glauber dynamics for the Curie Weiss and Ising model.

\tableofcontents

\section{Introduction}
Functional inequalities are powerful tools to quantify the trend to equilibrium of Markov semigroups and have a wide range important applications to the concentration of measure phenomenon and hypercontractivity. In a seminal work \cite{BAKEM} Bakry and \'Emery showed that a diffusion process on a Riemannian manifold whose generator is of the form 
$\mathcal{L}=\Delta + \nabla V\cdot \nabla$ satisfies the logarithmic Sobolev inequality if the pointwise bound 

\be\label{eq: Bak Em cond} \mathrm{Ric}+\mathrm{Hess} V \geq \kappa >0\ee 
holds.  Over the last three decades many profound results have been obtained in connection with the Bakry \'Emery condition \eqref{eq: Bak Em cond} and optimal transport, see the monographs \cite{villani2008optimal,bakry2013analysis}. Because of this success, considerable effort has been put into transferring the ideas and concepts of Bakry \'Emery theory to the setting of continuous time Markov chains. For example, Caputo et al. \cite{boudou2006spectral,caputo2009convex,dai2013entropy,samson2020entropic} developed a method based on a discrete analogous of Bochner's identity and obtained estimates on the spectral gap and entropy dissipation estimates for a large class of non local dynamics, whereas the more general problem of defining a notion of discrete Ricci curvature has been tackled in \cite{Olli,mielke2013geodesic,gozlan2014displacement}. In particular, the notion of entropic Ricci curvature put forward in \cite{maas2011gradient,erbar2012ricci} has deep implications in terms of functional inequalities. Explicit lower bounds for the entropic Ricci curvature in concrete examples have been recently obtained in \cite{erbar2015discrete,fathi2016entropic} and \cite{erbar2017ricci}. The exponential decay of general entropic functionals, called  $\phi$-entropies has been recently investigated in \cite{jungel2017discrete}, whereas functional inequalities for non linear Markov chains are the object of the preprint \cite{erbar2019entropic}. In this article we develop a probabilistic approach to establish convex Sobolev inequalites and quantify the exponential decay of $\phi$-entropies for continuous time Markov chains via the Bakry \'Emery method. Moreover, we apply this method on specific model examples and obtain explicit lower bounds. 
\paragraph*{Discrete convex Sobolev inequalities} In order to introduce $\phi$-entropies and discrete convex Sobolev inequalities, we consider a continuous time Markov chain on a countable state space $\Omega$, whose infinitesimal generator $\mathcal{L}$ takes the form

\be\label{def: Markov gen}\cL f(\eta) =  \sum_{\sigma\in G} c(\eta,\sigma) \nabla_\sigma f(\eta),  \ee

where $G$ is a collection of maps $\sigma:\Omega\longrightarrow\Omega$ called moves, $c:\Omega\times G\rightarrow \R_{\geq 0}$ are the transition rates and $\nabla_{\sigma}f(\eta)$ denotes the discrete gradient $\nabla_{\sigma}f(\eta)=f(\sigma\eta)-f(\eta)$. We will also assume that the Markov chain is reversible and denote  $\inv$ the reversible measure. Given a convex function $\phi:\R_{\geq0}\longrightarrow\R_{\geq0}$ and a positive function $f:\Omega\longrightarrow\R_{>0}$ the $\phi$-entropy $\cH^{\phi}(f|\inv)$ of $f$ is defined as follows:

\be\label{def: entropy}
\cH^{\phi}(f|\inv) = \sum_{\eta\in\Omega} \phi(f)(\eta) \inv(\eta) -  \phi \Big(\sum_{\eta\in\Omega} f(\eta) \inv(\eta)\Big).
\ee
In this work we are interested in estimating the best constant $\kappa_{\phi}$ such that the estimate

\be\label{phi entropy decay}
\cH^{\phi}(S_tf|\inv) \leq \exp(-\kappa_{\phi}t)\cH^{\phi}(f|\inv)
\ee
holds uniformly on $f>0$ and $t>0$. In the above, we denoted by $S_t$ the Markovian semigroup generated by $\mathcal{L}$. It is well known that \eqref{phi entropy decay} is equivalent to the \emph{convex Sobolev inequality}
\be\label{convex sob}
\forall f>0, \quad \kappa_\phi \cH^{\phi}(f|\inv) \leq \cE(\phi'(f),f),
\ee
where $\cE(f,g)$ is the Dirichlet form
\bes \cE(f,g)=-\sum_{\eta\in\Omega}f(\eta)\,(\cL g)(\eta)\, \inv(\eta).\ees
The family of convex Sobolev inequalities is quite rich. Indeed, defining 
\bes
\phi_{\alpha}(a)=\begin{cases}\frac{1}{\alpha-1}(a^\alpha-a)-a+1, \quad & \mbox{if $\alpha\in(1,2]$}\\
a \log a - a+1, \quad & \mbox{if $\alpha=1$}
\end{cases}
\ees
we get that \eqref{convex sob} is the \emph{modified logarithmic Sobolev inequality} (henceforth MLSI) for $\alpha=1$,
\be\label{log sob}
\forall f>0, \quad \kappa_1 \cH^{\phi_{1}}(f|\inv) \leq \cE(\log f,f).
\ee
For $\alpha=2$ we recover the \emph{Poincaré inequality}, whereas for $\alpha\in(1,2)$ we find the family of (discrete) \emph{Beckner inequalities} \cite{beckner1975inequalities,beckner1989generalized}.

\be\label{beckner ineq}
\forall f>0, \quad \kappa_\alpha \cH^{\phi_{\alpha}}(f|\inv) \leq \frac{\alpha}{\alpha-1}\cE(f^{\alpha-1},f).
\ee

For diffusions on a Riemannian manifold it is known \cite{arnold2001convex} that \eqref{convex sob} holds with $\kappa_{\phi}=\kappa$ if the Bakry \'Emery condition \eqref{eq: Bak Em cond} is satisfied, $\phi$ is convex and $\frac{1}{\phi''}$ is concave. Our strategy for establishing \eqref{phi entropy decay} and \eqref{convex sob} for Markov chains follows the original idea of \cite{BAKEM}, that is to prove the stronger convexity estimate

\be\label{bakem gamma 2}
\frac{\De^2}{\De t^2} \cH^{\phi}(S_tf|\inv) \geq \kappa_{\phi} \cE(\phi'(S_tf),S_tf),
\ee
To see why \eqref{bakem gamma 2} implies \eqref{convex sob} we recall that, at least formally we have

\bes
\frac{\De}{\De t} \cH^{\phi}(S_tf|\inv) = -\cE(\phi'(S_tf),S_tf).
\ees

Therefore, we obtain from \eqref{bakem gamma 2} and Gronwall's lemma that $\frac{\De}{\De t} \cH^{\phi}(S_tf|\inv) \longrightarrow0$. From this, \eqref{convex sob} follows integrating \eqref{bakem gamma 2} over $[t,\infty)$ provided that $\cH^{\phi}(S_tf|\inv)\longrightarrow 0$. In the continuous setting \eqref{bakem gamma 2} is obtained via Bochner's identity and pointwise comparison between the so called $\Gamma$ and $\Gamma_2$ operators. For Markov chains comparing first and second derivative of the entropy has proven to be quite challenging, and the picture is not fully clear yet.
\paragraph{Probabilistic approach to convex entropy decay}  In this work, we develop a method for establishing \eqref{bakem gamma 2} based on the notion \emph{coupling rates}. In order to define coupling rates, it is convenient to augment the set $G$ with a null element $\zero$ and set $G^*=G\cup\{\zero\}$.
 
\begin{mydef}\label{def couprates}
Let $\eta,\breta\in\Omega$ and $\cL$ as in \eqref{def: Markov gen}. We say that $\couprate(\eta,\breta,\cdot,\cdot):G^*\times G^* \longrightarrow \R_{\geq 0}$ are coupling rates for $(\eta,\breta)$ if

\begin{align}\label{eq:couprates}
\forall \gamma\in G, \quad \sum_{\bgamma \in G^*}\couprate(\eta,\breta,\gamma,\bgamma) = c(\eta,\gamma),\\
\nonumber \forall \bgamma\in G, \quad \sum_{\gamma \in G^*}\couprate(\eta,\breta,\gamma,\bgamma) = c(\breta,\bgamma).
\end{align}

\end{mydef}
 
If coupling rates are available for any pair $(\eta,\breta)$ then one can define a Markov generator $\cop$ acting on  $F:\Omega\times \Omega \longrightarrow \R$ as follows
\bes
\cop F (\eta,\breta) =\sum_{\gamma,\bgamma\in G^*} \couprate(\eta,\breta,\gamma,\bgamma) \coupgrad_{\gamma,\bgamma}F(\eta,\breta),
\ees
where  $\coupgrad_{\gamma,\bgamma}F(\eta,\breta):= F(\gamma\eta,\bgamma\breta)-F(\eta,\breta)$. A Markov chain on $\Omega\times\Omega$ with generator $\cop$ started at $(\eta,\breta)$ indeed realizes a coupling of a Markov chain with generator $\cL$ started at $\eta$ and of a Markov chain with generator $\cL$ started at $\breta$. For diffusions on Riemannian manifolds, it is well known that the fundamental gradient estimate \cite[Thm. 3.2.4]{bakry2013analysis} can be obtained with a coupling argument, see \cite{wang1997estimation}. Concerning Markov chains, we recall that couplings are a fundamental ingredient in the notion of coarse Ricci curvature \cite{Olli}. However, this notion is not known to imply neither \eqref{bakem gamma 2} nor \eqref{convex sob}. Finally, we remark that Chen has obtained in a series of paper (see for instance \cite{mufa1994optimal}) bounds on the spectral gap for birth and death chains by means of coupling arguments and that in the recent preprint \cite{hermon2019entropy} couplings are employed to obtain MLSI for inhomogeneous zero range processes using the so called martingale method. We conclude this introductory section summarizing the main contributions of this work.

\bei 
\item We propose a probabilistic alternative to the discrete Bochner identities of \cite{caputo2009convex}, upon which a large part of the results about MLSI recently obtained in connection with the Bakry \'Emery method rely. The notion of ``admissible function'' \cite[Def. 2.3]{dai2013entropy} is replaced by that of coupling rates. Although there is no blackbox for producing an efficient coupling in view of obtaining \eqref{bakem gamma 2}, there are some general guidelines. In particular, as one may expect, it is often convenient to construct the coupling rates in such a way that the associated Markov chain on $\Omega \times \Omega$ reaches as quickly as possible the diagonal $\{ (\eta,\eta):\eta\in \Omega\}$ and if it starts from the set $\{\eta,\breta: \exists \, \sigma\in G \,  \text{s.t.}  \, \breta=\sigma\eta \}$, it never leaves it. Therefore, the method is quite robust and could be used to analyse a wider class of models than those studied here.

\item A cornerstone result of Bakry \'Emery theory \cite{BAKEM} asserts that strongly log-concave probability measures on $\RD$ satisfy the logarithmic Sobolev inequality with a positive constant. This powerful geometric criterion is non perturbative, in the sense that it is satisfied by probability measures that may be \emph{far} from being product measures. On the contrary, most results for continuous time Markov chains are perturbative in spirit, ensuring positive lower bounds on the MLSI constant only if the interaction is small and $\inv$ is almost a product measure. In light of these observations, it is very natural to seek for non perturbative sufficient conditions on the generator of a continuous time Markov chain on $\N^d$ implying MLSI. To the best of our knowledge, such results have only been obtained for $d=1$, with the exception of some two dimensional examples treated in \cite{dai2013entropy}. It turns out that the use of coupling rates enables to lift the obstacles that have limited non perturbative criteria to the one dimensional setup and we shall propose at Theorem \ref{thm: LSI int random walks} below a sufficient non perturbative condition for MLSI and general convex Sobolev inequalities that is valid for any value of $d$. As a corollary, we obtain that multiplying a multidimensional Poisson distribution by a density of the form $\exp(-V)$ yields a probability measure satisfying MLSI if a local condition at the origin holds and the Hessian of the potential $V$ has non negative entries. This creates a curious parallelism with the above mentioned result for probability measures on $\RD$, where it is the non negativity of $\mathrm{Hess}\, V$ as a quadratic form that plays an essential role.

\item For interacting random walks, we show at Theorem \ref{thm: wass contr} that the sufficient condition for MLSI and convex Sobolev inequalities proposed at Theorem \ref{thm: LSI int random walks} is equivalent to an exponential contraction estimate for the Wasserstein distance along the semigroup generated by $\cL$. For diffusions on a Riemannian manifold it is known that the best constant in the logarithmic Sobolev inequality is at least as good as the best constant in the exponential contraction of the Wasserstein distance, see \cite{von2005transport} for example.
This fundamental result served as an inspiration for the notion of coarse Ricci curvature \cite{Olli} and it is a natural question to ask whether it admits a counterpart in the setting of continuous time Markov chains. To the best of our knowledge, this question has remained unanswered so far and Theorem \ref{thm: wass contr} settles it when the state space is $\N^d$. It is reasonable to expect that the conclusions of Theorem \ref{thm: wass contr} hold in a broader setup, for instance in that of section \ref{sec: spin}.

\item The proposed method provides with a unified framework for the study of general convex Sobolev inequalities, MLSI and Beckner inequalities. The literature about convex Sobolev inequalities for Markov chains is not abundant, see \cite{jungel2017discrete,bobkov2006modified}. Therefore, in many of the examples we analyse, the lower bounds on $\kappa_{\phi}$ that we obtain seem to be new. Concerning MLSI, we can sometimes improve on the best known estimates for $\kappa_1$ we are aware of, see sections \ref{sec: curie}, \ref{sec: Ising} and \ref{sec: hardcore}.
\eei
 
 \paragraph{Organization} In section \ref{sec: method} we state the basic assumptions and outline the method in an abstract setup. In section \ref{sec: interacting random walks} we present a general criterion that applies in particular to  interacting random walks. Moreover, we provide an interpretation of the lower bounds in terms of Wasserstein contraction. Another criterion is given in section \ref{sec: spin} that covers many classical spin systems. Section \ref{sec: review} deals with some classical models well studied in the literature: Bernoulli-Laplace models, hardcore models and zero range dynamics on the complete graph.

\section{Coupling rates and convex entropy decay}\label{sec: method}

In this section we state our main assumptions and give some simple but rather general results on how to use coupling rates to obtain convexity estimates for the evolution of entropic functionals. 

\subsection{Setup and main assumptions}
Given a state space $\Omega$ that is at most countable, a finite set of moves $G$, and non negative transition rates $c(\eta,\sigma)$ we consider the formal generator \eqref{def: Markov gen}. We make the following basic assumption.

\begin{itemize}
    \item[\namedlabel{H0}{\textbf{(H0)}}] The set $G$ is finite. $\cL$ is irreducible and admits an invariant probability measure $\inv\in\cP(\Omega)$ that satisfies
    $$ \sum_{\eta\in\Omega,\sigma\in G} c(\eta,\sigma)\inv(\eta)<+\infty. $$
\eei

It is well known, see for instance \cite{NORRIS}, that under \ref{H0} the invariant measure is unique and for any initial $\eta\in\Omega$ there exists a continuous time Markov chain $(X_t)_{t\geq0}$ whose infinitesimal generator is $\cL$ and such that $X_0=\eta$. Moreover, $(X_t)_{t\geq 0}$ is non-explosive. Following closely \cite{dai2013entropy} we also assume that $\inv$ is reversible for $\cL$ and that each move admits an ``inverse".

\bei 
\item[\namedlabel{H1}{\textbf{(H1)}}] There exists an involution
    \begin{align*}
        G \longrightarrow G\\
        \sigma\mapsto \sigma^{-1}
    \end{align*}
    
    such that $\sigma^{-1}(\sigma(\eta))=\eta$ holds whenever $\inv(\eta)c(\eta,\sigma)>0$ and

\be\label{eq rev}
\quad \sum_{\substack{\eta\in\Omega \\\sigma\in G}}F(\eta,\sigma)c(\eta,\sigma)\inv(\eta) = \sum_{\substack{\eta\in\Omega \\\sigma\in G}}F(\sigma\eta,\sigma^{-1})c(\eta,\sigma)\inv(\eta)
\ee

 holds for all bounded $F:\Omega\times G\longrightarrow \R$. 

\end{itemize}

Next, we shall define the functional inequalities that are the main object of interest of this paper. To avoid having to discuss the domain of $\cL$ and of the associated Dirchlet form, we begin by assuming that $\Omega$ is finite. In this case the Dirichlet form $\cE(f,g)$ can be defined for any pair of real valued functions $f$ and $g$ as

\be\label{dir form}
\cE(f,g) := -\sum_{\eta\in\Omega}g(\eta)(\cL f)(\eta) \inv(\eta).
\ee

It is well known (see for instance \cite[Eq. (2.12)]{dai2013entropy}) that under \ref{H1} we can rewrite $\cE(f,g)$ using \eqref{eq rev} as follows
 
 \be\label{dir form 2}
\cE(f,g)= \frac{1}{2} \sum_{\substack{\eta\in\Omega\\\sigma\in G}}c(\eta,\sigma)\nabla_{\sigma}f(\eta) \nabla_{\sigma}g(\eta) \, \inv(\eta). 
 \ee
 
Therefore, recalling the definition \eqref{def: entropy}  of $\phi$-entropy $\cH^{\phi}(\cdot|\inv)$ we have that for a given convex function $\phi:\R_{\geq 0}\longrightarrow \R_{\geq 0}$, the convex Sobolev inequality \eqref{convex sob} holds with constant $\kappa_{\phi}$ if and only if 

\be\label{eq: convex sob rev}
\forall f>0, \quad \cH^{\phi}(f|\inv) \leq \frac{\kappa_{\phi}}{2} \sum_{\substack{\eta\in\Omega\\\sigma\in G}} \nabla_{\sigma } (\phi'\circ f) \nabla_{\sigma}f(\eta)\, c(\eta,\sigma)   \inv(\eta).
\ee

Note that the convexity of $\phi$ makes sure that the right hand side of \eqref{eq: convex sob rev} is well defined even when $\Omega$ is not finite but countable and we shall use \eqref{eq: convex sob rev} as a definition of convex Sobolev inequality for countable state spaces. If $\Omega$ is finite, the fact that \eqref{eq: convex sob rev} implies the entropy dissipation estimate \eqref{phi entropy decay} is an immediate consequence of  

\bes
\forall t>0, \quad \frac{\De}{\De t} \cH^{\phi}(S_tf|\inv) = -\cE(\phi'(S_tf),S_tf).
\ees
When $\Omega$ is countable, some extra care has to be used as the above relation may not be valid for all positive $f$. For MLSI ($\phi(a)=a\log a -a +1$) and under hypothesis \ref{H0},\ref{H1} the validity of \eqref{phi entropy decay} is covered by  \cite[Prop 2.1]{dai2013entropy}. Some minor modifications of the argument therein cover the case of a general convex $\phi$. For the sake of brevity, we do not provide details here. In this article, we seek for conditions implying the convex Sobolev inequality that can be read directly off the generator $\cL$. In view of \eqref{eq: convex sob rev}, it is convenient to introduce the function $\Phi$ defined by

\bes
  \Phi: \R^2_{>0} \longrightarrow \R_{>0}, \quad \Phi(a,b):=(\phi'(b)-\phi'(a))(b-a).
 \ees
 A natural assumption for our method to work is the following
\bei
  \item[\namedlabel{H2}{\textbf{(H2)}}] $\phi$ is convex and the function
 $\Phi$ is also convex.
\eei

When $\phi=\phi_{\alpha}$, the function $\Phi$ is denoted $\Phi_{\alpha}$. We will show at Lemma \ref{lem: reny admissible}  that $\Phi_{\alpha}$ satisfies \ref{H2}.

\subsection{Coupling rates and second derivative of the entropy}

The goal of this section is to show how one can use coupling rates to organize the terms originating from differentiating $\cE(\phi'(S_t f),S_t f)$ and find appropriate lower bounds in view of establishing \eqref{bakem gamma 2}. We begin by recording some useful properties of $\Phi_{\alpha}$ that we shall use to obtain Beckner inequalites and MLSI. From now on, for a differentiable function $(a,b)\mapsto \Phi(a,b)$ we denote by $\rmD \Phi(a,b)$ the Jacobian, i.e. the $1\times 2$ matrix $[\partial_a \Phi(a,b),\partial_b\Phi(a,b)]$. We also use the notation $\cdot$ for the standard matrix-vector product.

 \begin{lemma}\label{lem: reny admissible}
 Let $\alpha\in[1,2]$. Then $\phi_{\alpha}$ satisfies \ref{H2}. Moreover
 \bei
 \item If $a,b,a',b'>0$ are such that $a'=b'$ we have
 \be\label{eq:Beck improve}
\Phi_{\alpha}(a',b')-\Phi_{\alpha}(a,b)- \mathrm{D}\Phi_{\alpha} (a,b) \cdot \colvec{a'-a}{b'-b} \geq (\alpha-1) \Phi_{\alpha}(a,b). 
\ee
\item For all $a,b>0$ we have 
\begin{align}\label{eq: MLSI improve}
\nonumber &\Phi_{1}(a,a)-\Phi_{1}(a,b)- \mathrm{D}\Phi_{1} (a,b) \cdot \colvec{0}{a-b}  \\
&+ \Phi_{1}(b,b)-\Phi_{1}(a,b)- \mathrm{D}\Phi_{1} (a,b) \cdot \colvec{b-a}{0} \geq 2\Phi_1(a,b).
\end{align}

\eei
 \end{lemma}
 
 We defer the proof of this algebraic lemma to the appendix. 
 \begin{remark} Assumption \ref{H2} is different from that of \cite{jungel2017discrete}, where it is assumed that $(a-b)/(\phi'(a)-\phi'(b))$ is concave. This assumption implies in particular that $a\mapsto \frac{1}{\phi''(a)}$ is concave, which is the classical hypothesis used for diffusions on manifolds. However, in order to go beyond the study of one dimensional birth and death processes, an extra homogeneity assumption has to be added there.
 \end{remark}
 
 In the next lemma we observe that coupling rates can be used to organize the terms coming from $\frac{\De}{\De t} \cE(\phi'(f_t),f_t)|_{t=0}$ and find a first general upper bound. For the next lemma and in all what follows it is convenient to define $\supp \subseteq \Omega \times G$ as
 
 \bes
 \supp  = \{ (\eta,\sigma)\in \Omega \times G : c(\eta,\sigma)>0 \}.
 \ees
 
 Moreover, we recall that $G^*$ is the set of moves augmented with the null-move $\zero$, i.e. $G^*=G\cup\{\zero\}$
 and $\zero\eta=\eta$ for all $\eta\in\Omega$. To streamline proofs and avoid technicalities we assume that $\Omega$ is finite, although we believe this assumption not to be essential. Finally, to ease notation we shall write $f_t$ instead of $S_tf$.

 \begin{lemma}\label{lem: coupling organize}
 Assume \ref{H0}-\ref{H2} and let $\{\couprate(\eta,\sigma\eta,\cdot,\cdot)\}_{(\eta,\sigma)\in\supp}$ be coupling rates. For all $f>0$ define $\fphi$ as
 \bes
  \fphi:\Omega\times\Omega\longrightarrow\R_{\geq 0}, \qquad \fphi(\eta,\breta)= \Phi(f(\eta),f(\breta)).
 \ees
 We have:

\be\label{eq: organize}
\frac{\De}{\De t} 2 \cE(\phi'(f_t),f_t)\big|_{t=0}= \sum_{\substack{(\eta,\sigma)\in\supp\\ \gamma,\bgamma\in G^*}} c(\eta,\sigma)\couprate(\eta,\sigma\eta,\gamma,\bgamma)\mathrm{D}\Phi(f(\eta),f(\sigma\eta)) \cdot \colvec{\nabla_{\gamma}f(\eta)}{\nabla_{\bgamma}f(\sigma\eta)}\inv(\eta).
\ee
Consequently,
\begin{align}\label{eq: basic upper bound}
  \nonumber\frac{\De}{\De t} 2\cE(\phi'(f_t),f_t)\big|_{t=0} &\leq \sum_{\substack{(\eta,\sigma)\in\supp\\ \gamma,\bgamma\in G^*}}c(\eta,\sigma)\couprate(\eta,\sigma\eta,\gamma,\bgamma) \coupgrad_{\gamma,\bgamma} \fphi(\eta,\sigma\eta)\, \inv(\eta)\\
 & - \sum_{\substack{(\eta,\sigma)\in\supp\\ \gamma,\bgamma\in G^* \\ \gamma\eta=\bgamma\sigma\eta}}c(\eta,\sigma)\couprate(\eta,\sigma\eta,\gamma,\bgamma) \Big(\coupgrad_{\gamma,\bgamma} \fphi(\eta,\sigma\eta)- \mathrm{D}\Phi(f(\eta),f(\sigma\eta)) \cdot \colvec{\nabla_{\gamma}f(\eta)}{\nabla_{\bgamma}f(\sigma\eta)}  \Big) \inv(\eta).
 \end{align}
 \end{lemma}

  \begin{proof}
By definition of $\Phi$ and \eqref{dir form 2} we have for all $f>0$

\bes
2\cE(\phi'(f),f)= \sum_{(\eta,\sigma)\in\cS} c(\eta,\sigma)\Phi(f(\eta),f(\sigma\eta)) \inv(\eta).
\ees
Therefore,
\be\label{eq: second der first form}
\frac{\De}{\De t} 2\cE(\phi'(f_t),f_t)\big|_{t=0} = \sum_{(\eta,\sigma)\in\supp}c(\eta,\sigma) \mathrm{D}\Phi(f_t(\eta),f_t(\sigma\eta)) \cdot \colvec{\cL f(\eta)}{\cL f(\sigma\eta)}\inv(\eta).
\ee
From the definition of coupling rates \eqref{def couprates}, we get that for $(\eta,\sigma)\in\cS$,
\bes
\cL f(\eta) = \sum_{\gamma\in G}c(\eta,\gamma)\nabla_{\gamma}f(
\eta)= \sum_{\gamma,\bgamma\in G^*}\couprate(\eta,\sigma\eta,\gamma,\bgamma)\nabla_{\gamma}f(\eta).
\ees
Rewriting $\cL f(\sigma\eta)$ analogously and plugging the two resulting expression back into \eqref{eq: second der first form} we arrive at \eqref{eq: organize}. To derive \eqref{eq: basic upper bound}  we can first add and substract 
\bes \sum_{\substack{(\eta,\sigma)\in\supp\\ \gamma,\bgamma\in G^*}}c(\eta,\sigma)\couprate(\eta,\sigma\eta,\gamma,\bgamma) \coupgrad_{\gamma,\bgamma} \fphi(\eta,\sigma\eta)\, \inv(\eta)
\ees
in  \eqref{eq: second der first form}. We obtain the equivalent expression
\begin{align}\label{eq: basic upper bound proof}
  \nonumber\frac{\De}{\De t} 2\cE(\phi'(f_t),f_t)\big|_{t=0} &= \sum_{\substack{(\eta,\sigma)\in\supp\\ \gamma,\bgamma\in G^*}}c(\eta,\sigma)\couprate(\eta,\sigma\eta,\gamma,\bgamma) \coupgrad_{\gamma,\bgamma} \fphi(\eta,\sigma\eta)\, \inv(\eta)\\
 & - \sum_{\substack{(\eta,\sigma)\in\supp\\ \gamma,\bgamma\in G^*}}c(\eta,\sigma)\couprate(\eta,\sigma\eta,\gamma,\bgamma) \Big(\coupgrad_{\gamma,\bgamma} \fphi(\eta,\sigma\eta)- \mathrm{D}\Phi(f(\eta),f(\sigma\eta)) \cdot \colvec{\nabla_{\gamma}f(\eta)}{\nabla_{\bgamma}f(\sigma\eta)}  \Big) \inv(\eta).
 \end{align}
Since $\coupgrad_{\gamma,\bgamma}\fphi(\eta,\sigma\eta)=\Phi(f(\gamma\eta),f(\bgamma\sigma\eta))-\Phi(f(\eta),f(\sigma\eta))$, we deduce from the convexity of $\Phi$ that for all $(\eta,\sigma)\in \supp$ and all $\gamma,\bgamma\in G^*$ we have
\bes
\coupgrad_{\gamma,\bgamma} \fphi(\eta,\sigma\eta)- \mathrm{D}\Phi(f(\eta),f(\sigma\eta)) \cdot \colvec{\nabla_{\gamma}f(\eta)}{\nabla_{\bgamma}f(\sigma\eta)} \geq 0. 
\ees
Thus, we obtain the upper bound \eqref{eq: basic upper bound} by dropping in \eqref{eq: basic upper bound proof} all terms such that $\gamma\eta\neq\bgamma\sigma\eta$ .
\end{proof}

We now present a sufficient condition to obtain convex Sobolev inequalities. In there, we again assume for simplicity that $\Omega$ is finite although this is probably not necessary. Under this assumption, the proof that \eqref{bakem gamma 2} implies \eqref{eq: convex sob rev} is straightforward. At Corollary \ref{cor: countable state space} we provde a simple sufficient condition that allows to extend the results for finite to countable state spaces via a localization procedure. 

 \begin{prop}\label{cor: sufficient cond}
 Let $\Omega$ be finite, assume \ref{H0},\ref{H1},\ref{H2} and let $\{\couprate(\eta,\sigma\eta,\cdot,\cdot)\}_{(\eta,\sigma)\in\supp}$ be coupling rates.  If 
 \bei
 \item There exists $\kappa'\geq 0$ such that
 \be\label{eq: discrete gradient estimate}
\frac12\sum_{\substack{(\eta,\sigma)\in\supp\\\gamma,\bgamma\in G^*}}c(\eta,\sigma)\couprate(\eta,\sigma\eta,\gamma,\bgamma) \coupgrad_{\gamma,\bgamma} \fphi (\eta,\sigma\eta)\, \inv(\eta) \leq -\kappa' \cE(\phi'(f),f)
 \ee
 holds uniformly on $f>0$
\item  There exist $\kappa'',\kappa'''\geq 0$ such that
\be\label{def: kappa''}
\inf_{(\eta,\sigma)\in\cS} \min\{\couprate(\eta,\sigma\eta,\sigma,\zero),\couprate(\eta,\sigma\eta,\zero,\sigma^{-1})\}\geq \kappa''
\ee
and 
\be\label{def: kappa'''}
\inf_{(\eta,\sigma)\in\cS} \sum_{\substack{\gamma,\bgamma \in G^*\\\gamma\eta=\bgamma\sigma\eta}} \couprate(\eta,\sigma\eta,\gamma,\bgamma)\geq \kappa'''
\ee
hold.
\eei

Then
 \bei 
 \item[(i)] The convex Sobolev inequality \eqref{convex sob} holds with $\kappa_{\phi}=\kappa'$ for all $\Phi$ satisfying \ref{H2}. 
  \item[(ii)] The modified logarithmic Sobolev inequality \eqref{log sob} holds with $\kappa_{1}=\kappa'+2\kappa''$.
 \item[(iii)] For $\alpha \in (1,2]$, the discrete Beckner inequality \eqref{beckner ineq} holds with $\kappa_{\alpha}=\kappa'+(\alpha-1)\kappa'''$.
 \eei
 \end{prop}
 
 \begin{proof}
 We begin by proving $(i)$. Let $f>0$ and $\Phi$ satisfy \ref{H2} and consider the bound \eqref{eq: basic upper bound}. Using the convexity of $\Phi$ and \eqref{eq: discrete gradient estimate} we deduce that
 \bes
\frac{\De}{\De t} \cE(\phi'(f_t),f_t)\big|_{t=0} \leq -\kappa'\cE(\phi'(f), f).
 \ees
We have therefore established the Bakry \'Emery convexity estimate \eqref{bakem gamma 2} with $\kappa_{\phi}=\kappa'$, from which the desired conclusion follows. To prove $(ii)$ we observe that the convexity of $\Phi$  gives

\begin{align*}
&\sum_{\substack{(\eta,\sigma)\in\supp\\ \gamma,\bgamma\in G^* \\ \gamma\eta=\bgamma\sigma\eta}}c(\eta,\sigma)\couprate(\eta,\sigma\eta,\gamma,\bgamma) \Big(\coupgrad_{\gamma,\bgamma} f^{\phi_1}(\eta,\sigma\eta)- \mathrm{D}\Phi_1(f(\eta),f(\sigma\eta)) \cdot \colvec{\nabla_{\gamma}f(\eta)}{\nabla_{\bgamma}f(\sigma\eta)}  \Big) \inv(\eta)\\
\geq&\sum_{\substack{(\eta,\sigma)\in\supp}}c(\eta,\sigma)\couprate(\eta,\sigma\eta,\zero,\sigma^{-1}) \Big(\coupgrad_{\zero,\sigma^{-1}} f^{\phi_1}(\eta,\sigma\eta)- \mathrm{D}\Phi_1(f(\eta),f(\sigma\eta)) \cdot \colvec{0}{-\nabla_{\sigma}f(\eta)}  \Big) \inv(\eta)\\
+&\sum_{\substack{(\eta,\sigma)\in\supp}}c(\eta,\sigma)\couprate(\eta,\sigma\eta,\sigma,\zero) \Big(\coupgrad_{\sigma,\zero} f^{\phi_1}(\eta,\sigma\eta)- \mathrm{D}\Phi_1(f(\eta),f(\sigma\eta)) \cdot \colvec{\nabla_{\sigma}f(\eta)}{0}  \Big) \inv(\eta)\\
\geq& \kappa'' \sum_{(\eta,\sigma)\in\supp}c(\eta,\sigma)\Big(\coupgrad_{\zero,\sigma^{-1}} f^{\phi_1}(\eta,\sigma\eta)- \mathrm{D}\Phi_1(f(\eta),f(\sigma\eta)) \cdot \colvec{0}{-\nabla_{\sigma}f(\eta)} \\
+&  \coupgrad_{\sigma,\zero} f^{\phi_1}(\eta,\sigma\eta)- \mathrm{D}\Phi_1(f(\eta),f(\sigma\eta)) \cdot \colvec{\nabla_{\sigma}f(\eta)}{0}\Big)\inv(\eta) \\
 \geq& 4\kappa''\cE(\phi'(f),f)
\inv(\eta),
\end{align*}
where to obtain the last inequality we used \eqref{eq: MLSI improve} with $a=f(\eta),b=f(\sigma\eta)$. Combining this last estimate with \eqref{eq: discrete gradient estimate} in \eqref{eq: basic upper bound} yields \eqref{bakem gamma 2} with $\kappa_1=\kappa'+2\kappa''$, from which the desired conclusion follows. The proof of $(iii)$ is analogous. Indeed, \eqref{eq:Beck improve} with $a=f(\eta),b=f(\sigma\eta)$ and $a'=b'=\gamma\eta$ gives
\begin{align*}
&\sum_{\substack{(\eta,\sigma)\in\supp\\ \gamma,\bgamma\in G^* \\ \gamma\eta=\bgamma\sigma\eta}}c(\eta,\sigma)\couprate(\eta,\sigma\eta,\gamma,\bgamma) \Big(\coupgrad_{\gamma,\bgamma} f^{\phi_{\alpha}}(\eta,\sigma\eta)- \mathrm{D}\Phi_{\alpha}(f(\eta),f(\sigma\eta)) \cdot \colvec{\nabla_{\gamma}f(\eta)}{\nabla_{\bgamma}f(\sigma\eta)}  \Big) \inv(\eta)\\
&\geq (\alpha-1) \kappa'''\sum_{\substack{(\eta,\sigma)\in\supp\\ \gamma,\bgamma\in G^*}}c(\eta,\sigma) \Phi_{\alpha}(f(\eta),f(\sigma\eta))\inv(\eta)\\
=&2\kappa''' \cE(\phi'(f),f)
\end{align*}
Using this last estimate and \eqref{eq: discrete gradient estimate} in \eqref{eq: basic upper bound} yields \eqref{bakem gamma 2} with $\kappa_{\alpha}=\kappa'+(\alpha-1)\kappa'''$, from which the desired conclusion follows.
 \end{proof}
It is known \cite{dai2013entropy,jungel2017discrete} that the best constant $\kappa_{\alpha}$ in \eqref{beckner ineq} is such that $\kappa_{\alpha} \leq \kappa_2$ for all $\alpha\in[1,2]$, i.e. Beckner's inequalities and in particular MLSI are stronger than Poincar\'e inequality. When working in a continuous state space, using the fact that the operator $\Gamma$ is a derivation, more relations between the constants $\kappa_{\alpha}$ are known \cite[Sec. 2.8]{bakry2013analysis}. In particular, a Poincaré inequality implies a Beckner inequality for $\alpha\in (1,2]$ with a non-optimal constant. Because of the non locality of generators, this fact that does not carry over in a straightforward way to continuous time Markov chains and no general result in this direction is known to the author. We conclude this section adapting Proposition \ref{cor: sufficient cond} to countable state spaces. As we said above, we achieve this with a localization procedure. To describe it, let $\cL$ be a generator defined on a countable state space $\Omega$ such that \ref{H0},\ref{H1} are satisfied and let $\inv$ be the associated reversible probability measure. Next, consider an increasing sequence of finite subsets $(\Omega^n)_{n\in\N}\subseteq \Omega$ such that $\bigcup_{n\in\N} \Omega^n=\Omega$. For $\sigma\in G$ define the move $\sigma^n:\Omega^n\longrightarrow \Omega^n$  and the rates $c(\eta,\sigma^n)$ as follows:

\be\label{def: loc moves loc rates}
\forall\eta \in\Omega^n,\sigma\in G, \quad \sigma^n\eta= \begin{cases} \sigma\eta,& \mbox{if $\sigma\eta\in\Omega^n$}\\
\eta,& \mbox{otherwise.}
\end{cases}, \quad 
c(\eta,\sigma^n)=  c(\eta,\sigma). 
\ee
 Finally, we consider the generator $\cL^n$ on $\Omega^n$ given by
 \be\label{def: loc gen}
 \cL^n f(\eta) = \sum_{\sigma\in G} c(\eta,\sigma^n)\nabla_{\sigma^n}f(\eta).
 \ee
It is not hard to see that if we denote by $\inv^n\in\cP(\Omega^n)$ the conditioning of $\inv$ to $\Omega^n$, i.e. \[ \forall A\subseteq\Omega^n,\quad \inv^n(A) = \frac{1}{\inv(\Omega^n)} \inv(A),\] 
then $\inv^n$ is invariant for $\cL^n$ and \ref{H1} is satisfied with $(\sigma^n)^{-1}=(\sigma^{-1})^n$.
\begin{cor}\label{cor: countable state space}
Assume that $\cL$ satisfies \ref{H0},\ref{H1}, that $\Omega$ is countable and that $\phi$ satisfies \ref{H2}. Consider an increasing sequence of finite sets $(\Omega^n)_{n\in\N}$ such that $\Omega=\bigcup_{n\in\N}\Omega^n$ and define $\cL^n$ by \eqref{def: loc moves loc rates},\eqref{def: loc gen}. Moreover, assume that there exists $\kappa>0$ such that for all $n\in\N$ the convex Sobolev inequality
 \be\label{eq: localized convex sob}
\forall f>0, \quad \kappa\cH^{\phi}(f|\inv^n) \leq \frac12  \sum_{\substack{\eta\in\Omega^n \\ \sigma\in G }}c(\eta,\sigma^n)\Phi(f(\eta),f(\sigma^n\eta))\inv^n(\eta)
\ee
holds uniformly in $n$. Then the convex Sobolev inequality \eqref{eq: convex sob rev} holds with $\kappa_{\phi}=\kappa$. \end{cor}
 
 \begin{proof}
 Let $f:\Omega\longrightarrow\R$ be compactly supported and denote $\bar\Omega$ the support. If $\bar{n}$ is large enough so that$\Omega^n\supseteq \{ \sigma\eta: \eta\in \bar\Omega,\sigma\in G\}$ for all $n\geq\bar n$. Applying \eqref{eq: localized convex sob} to the restriction of $f$ to $\Omega^n$ and using the definition of the rates \eqref{def: loc moves loc rates} yields for all $n\geq\bar{n}$
\bes
\kappa\left[ \sum_{\eta\in\bar\Omega} \phi(f)(\eta)\frac{\inv(\eta)}{\inv(\Omega^n)}-  \phi\Big(\sum_{\eta\in\bar\Omega}f(\eta)\frac{\inv(\eta)}{\inv(\Omega^n)}\Big) \right]\leq \frac12 \sum_{\substack{\eta\in\bar\Omega \\ \sigma\in G }}c(\eta,\sigma)\Phi(f(\eta),f(\sigma\eta))\frac{\inv(\eta)}{\inv(\Omega^n)}.
\ees
Letting $n\rightarrow+\infty$ and recalling that $\nabla_{\sigma } (\phi'\circ f) \nabla_{\sigma}f(\eta)=\Phi(f(\eta),f(\sigma\eta))$ gives \eqref{eq: convex sob rev}. A standard approximation argument using compactly supported functions and monotone convergence allows to extend the result to a non compactly supported $f>0$.
 \end{proof}

In the remainder of this article, we use Proposition \ref{cor: sufficient cond} and Corollary \ref{cor: countable state space} to obtain lower bounds for the best constant in \eqref{eq: convex sob rev} in various concrete examples.

\section{Interacting random walks}\label{sec: interacting random walks}
 A cornerstone result of Bakry \'Emery thoery asserts that if we multiply the standard Gaussian distribution on $\RD$ by a log-concave density, then the resulting probability measure $\inv$ satisfies the logarithmic Sobolev inequality

\be\label{eq: langevin log sob}
\kappa \cH(f|\inv)\leq \int |\nabla \log f|^2(x) \inv(\De x) 
\ee
 with a constant $\kappa$ that is at least as large as the optimal constant for the Gaussian distribution. This result is a \emph{geometric} and non pertutbative sufficient condition implying the logarithmic Sobolev inequality. Indeed, log-concave probability measures are not necessarily close to product measures. On the lattice $\N^d$ the fundamental role played the Gaussian distribution on $\RD$ is taken up by the (multidimensional) Poisson distribution $\bm\mu_{\lambda}\in\cP(\N^d)$:
 
\bes
\bm\mu_{\lambda}(\eta_1,\ldots,\eta_d) = \prod_{i=1}^d\exp(\lambda)\frac{\lambda^{-\eta_i}}{\eta_i!}
\ees
where $1/\lambda>0$ is the intensity parameter. In analogy with Bakry \'Emery theory, one is lead to consider the following problem
\bei
\item Find \emph{non perturbative} conditions on $V:\N^d\rightarrow \R$  implying that $\inv=\exp(-V)\bm\mu_{\lambda}$ satisfies the convex Sobolev inequality with a positive constant.
\eei
In the above, by non perturbative we mean that the sought conditions do not necessarily imply that $V$ has to be small.
In the language of statistical mechanics, this means that we try to go beyond the high temperature/weak interaction regime. Of course, in order to give a meaning to the inequality \eqref{eq: convex sob rev} we need to first choose a generator $\cL$ for which $\inv$ is the reversible measure. Following the classical choice made in \cite{caputo2009convex},\cite{dai2013entropy}\cite{fathi2016entropic}, we begin by setting $\Omega=\N^d$,$G=\{\gamma^{\pm}_i: i=1,\ldots,d\}$ and recalling the standard notation  $(\mathbf{e}_i)_{i=1,\ldots,d}$ for the canonical basis of $\N^d$. Next, we define the formal generator $\cL$ as

\be\label{def: gen int rw}
\cL f (\eta) = \sum_{i=1}^d c(\eta,\gamma^{+}_i) \nabla_{i}^+ f(\eta) +c(\eta,\gamma^{-}_i)\nabla_{i}^- f(\eta),
\ee
where for all $\eta\in \N^d$ and $1\leq i\leq d$
\bes
\gamma^{+}_i\eta=\eta+ \mathbf{e}_i,\quad \gamma_i^-\eta= \eta-\mathbf{e}_i\mathbf{1}_{\eta_i>0}, \quad \nabla_{i}^{\pm}f(\eta)=\nabla_{\gamma^{\pm}_i}f(\eta),
\ees
and 
\be\label{rates: int random walks}
c(\eta,\gamma^+_i) = \exp(-\nabla_{i}^+V(\eta)), \quad c(\eta,\gamma^{-}_i) = \lambda\eta_i.
\ee
We shall give in the sequel some natural conditions on $V$ ensuring that \ref{H0} and \ref{H1} hold.

\paragraph{Contraction of the Wasserstein distance and functional inequalities}
Another fundamental result for diffusions on manifolds is that Wasserstein contraction at rate $\kappa$ implies the logarithmic Sobolev inequality (and others) with the same constant. To illustrate this, consider $V:\RD\rightarrow\R$ and denote by $S_t(x)=\bbE[f(X^x_t)]$ the semigroup generated by the Kolmogorov diffusion $(X^x_t)_{t\geq0}$ 

\be\label{Langevin}
\De X^x_t = -\nabla V(X^x_t)\De t + \sqrt{2}\De B_t, \quad X^x_0 \sim x.
\ee

We say that the $p$-Wasserstein distance $W_{p}(\cdot,\cdot)$ contracts at rate $\kappa$ if for all $\mu,\nu\in\cP(\RD)$ with finite $p$-th moment and all $t>0$ we have

\be\label{Wass contr}
W_p(\mu_t,\nu_t) \leq \exp(-\kappa t) W_p(\mu,\nu)
\ee

where $\mu_t =\mu S_t$ and $\nu_t=\nu S_t$ are the laws of the Kolmogorov diffusion at time $t$ started at $\mu$ and $\nu$ respectively. For diffusion processes on manifolds of the form \eqref{Langevin} it is known that for any fixed $p\geq 1$ the contraction estimate \eqref{Wass contr} is equivalent to the Bakry \'Emery condition \eqref{bakem gamma 2} and therefore implies the logarithmic Sobolev inequality \eqref{log sob}, see \cite{von2005transport} for instance. In particular, when the underlying manifold is $\RD$, Wasserstein contraction is equivalent to $\kappa$-convexity of $V$. 
Going back to the lattice $\N^d$, the following question arises naturally:
\bei 
\item Is there a relation between the contraction properties of $\cL$ as defined in \eqref{rates: int random walks} and the best constant in the convex Sobolev inequality \eqref{eq: convex sob rev}?
\eei

The notion of coarse Ricci curvature \cite{Olli} is based on the contraction of the $W_1$ distance. However, this notion is not known to imply MLSI. In the next two subsections we bring some answers to the questions raised here.

\subsection{A sufficient non perturbative condition}

The main result of this section is Theorem \ref{thm: LSI int random walks}, which contains a sufficient condition for \eqref{eq: convex sob rev} to hold. We begin by stating its main assumptions that are general enough to include \eqref{rates: int random walks} as a particular case. For two given potentials $V^{-},V^+:\N^d\rightarrow\R$ we construct the formal generator $\cL$  by

\be\label{def: gen int rw double pot}
\cL f (\eta) = \sum_{i=1}^d c(\eta,\gamma^{+}_i) \nabla_{i}^+ f(\eta) +c(\eta,\gamma^{-}_i)\nabla_{i}^- f(\eta),
\ee
where for all $\eta\in \N^d$ and $1\leq i\leq d$
\be\label{rates: int random walks double pot}
c(\eta,\gamma^+_i) = \exp(-\nabla_{i}^+V^+(\eta)), \quad c(\eta,\gamma^{-}_i) =  \exp(-\nabla_{i}^-V^-(\eta)).
\ee
When $V^+=V$ and $V^-(\eta)= \sum_{i=1}^d\log(\lambda) \eta_i+ \log (\eta_i!)$ we recover \eqref{rates: int random walks}.
We make the following hypothesis, see figure \ref{figure:s kappa condition} for an explanation.

\begin{itemize}
  \item[\namedlabel{H3.3}{\textbf{(H3.3)}}]  For all $\eta\in\Omega,1\leq i\leq d$ we have $\kappa^+(\eta,i)\geq 0$ and $\kappa^-(\eta,i)\geq 0$, where
  \be\label{def: kappa+}
  \kappa^+(\eta,i)= -\nabla_{i}^+c(\eta,\gamma^+_i) -\sum_{\substack{\bgamma \in G\\\bgamma\neq \gamma^+_{i},\gamma^-_i}}  \max\{\nabla_{i}^+ c(\eta,\bgamma),0 \},
  \ee
  
  and 
  \be\label{def: kappa-}
  \kappa^-(\eta,i)=\nabla_i^+c(\eta,\gamma_i^-)-\sum_{\substack{\gamma \in G\\\gamma\neq \gamma^+_{i},\gamma^-_i}}  \max\{-\nabla_{i}^+ c(\eta,\gamma),0 \}.
  \ee
  
\end{itemize}

\begin{theorem}\label{thm: LSI int random walks} Let $V^-,V^+:\N^d\rightarrow\R$ be given and the generator $\cL$ be defined by \eqref{def: gen int rw double pot},\eqref{rates: int random walks double pot}. Moreover, assume that \ref{H3.3} holds. Then \ref{H0},\ref{H1} hold with $\inv=\frac{1}{Z}\exp(-V^+-V^-)$, where $Z$ is the normalization constant. If we define

\be\label{kappa log concave condition}
\kappa= \inf_{\substack{\eta\in\N^d \\ 1\leq i\leq d}}   \kappa^+(\eta,i)+ \kappa^-(\eta,i)
\ee

then the following holds

 \item[(i)] For any $\phi$ satisfying \ref{H2} the convex Sobolev inequality \eqref{eq: convex sob rev} holds with with $\kappa_{\phi}=\kappa$. In particular MLSI holds with $\kappa_1=\kappa$. 
 
 \item[(ii)] For $\alpha \in (1,2]$, the Beckner inequality \eqref{beckner ineq} holds with $\kappa_{\alpha}=\alpha\kappa$.
\end{theorem}

To prove this Theorem, we show \eqref{eq: discrete gradient estimate} and \eqref{def: kappa'''} with $\kappa'=\kappa'''=\kappa$. In order to do so, we shall construct appropriate coupling rates: we refer to \eqref{figure: coupling support} to illustrate some properties of these rates.

\paragraph{Comparison with existing literature}
Concerning MLSI and the spectral gap, the results of Theorem \ref{thm: LSI int random walks} are well known for $d=1$. Comparable results can be found e.g. in \cite{fathi2016entropic,mielke2013geodesic,caputo2009convex,boudou2006spectral}. For Beckner inequalities and $d=1$, we refer to \cite{jungel2017discrete}. When $d>1$ much less appear to be known. For MLSI and Poincaré inequalities perturbative sufficient conditions are given in \cite{boudou2006spectral,dai2013entropy}. 
In \cite{dai2013entropy}, a non perturbative two dimensional example is also treated. Erbar et al. gave in \cite[Thm 3.9]{erbar2017ricci} a general abstract sufficient condition implying positive lower bounds for the entropic Ricci curvature \cite{erbar2012ricci,maas2011gradient}. It can be checked that this criterion, when applied to the setting \eqref{rates: int random walks double pot} provides a bound bounds for $\kappa_1$ that is worse or the same as what Theorem \ref{thm: LSI int random walks} gives, and in some cases it may give no positive lower bounds. However, although the authors only apply their result in the weak interaction/high temperature regime, it seems that its validity extends to the non perturbative setup. In particular, it could be used to provide lower bounds for some of the examples we are going to present at section \ref{sec: int rw examples}. Since the hypothesis of \cite[Thm 3.9]{erbar2017ricci} require some work to be checked and it is not obvious how to come up with a $\cL$ satisfying them, it could be interesting to try to use the probabilistic intuition behind coupling rates to construct candidate $\cL$ to which this criterion applies. For $d>2$ the results of Theorem \ref{thm: LSI int random walks} about Beckner's inequalities and general convex Sobolev inequalities seem to be new.
\begin{figure}[ht]
  \centering
  \begin{tikzpicture}
   \coordinate (Origin)   at (0,0);

    \clip (-3,-0.2) rectangle (7cm,7.5cm); 
    \pgftransformcm{1.8}{0}{0}{1.8}{\pgfpoint{0cm}{0cm}}
    \coordinate (eta) at (0,2);
    \coordinate (breta) at (2,2);
    \draw[style=help lines,dashed] (-14,-14) grid[step=2cm] (14,14);
    \foreach \x in {-7,-6,...,7}{
      \foreach \y in {-7,-6,...,7}{
        \node[draw,circle,inner sep=2pt,fill] at (2*\x,2*\y) {};
      }
    }
    \node at (-0.12,1.8) {$\eta$};
    \node at (1.75,1.8) {$\gamma_i^+\eta$};
    \draw [ultra thick, black] (0,2)
        -- (0.6,2) node [below] {};
    \draw [ultra thick,-latex,green] (0.6,2)
        -- (1.4,2) node [below left] {$-\nabla_i^+c(\eta,\gamma_i^+)$};
     \draw [ultra thick,-latex, black] (0,2)
        -- (0,2.4) node [below] {};
     \draw [ultra thick, black] (2,2)
        -- (2,2.4) node [below] {};
    \draw [ultra thick,-latex,red] (2,2.4)
        -- (2,2.75) node [right] {$\max\{\nabla_i^+c(\eta,\gamma_j^+),0\}$};  
    \draw [ultra thick,-latex,black] (2,2)
        -- (2.6,2) node [right] {};  
    \draw [ultra thick, black] (2,2)
        -- (2,1.3) node [below] {};
    \draw [ultra thick,-latex,red] (2,1.3)
        -- (2,1.01) node [right] {$\max\{\nabla_i^+c(\eta,\gamma_j^-),0\}$};  
     \draw [ultra thick, -latex, black] (0,2)
        -- (0,1.3) node [below] {};
  \end{tikzpicture}
  \caption{The condition $\kappa^+(\eta,i)\geq0$ imposes that the length of the green arrow is at least as much as the total length of the red arrows. The coupling interpretation of this condition is that the random walker starting at $\eta$ can use his/her larger probability to make the $\gamma_i^+$ move in order to run after the walker starting at $\gamma_i^+\eta$ whenever he/she tries to get at distance two from $\eta$ using the moves $\gamma_j^+,\gamma_j^-$. }
  \label{figure:s kappa condition}
\end{figure}
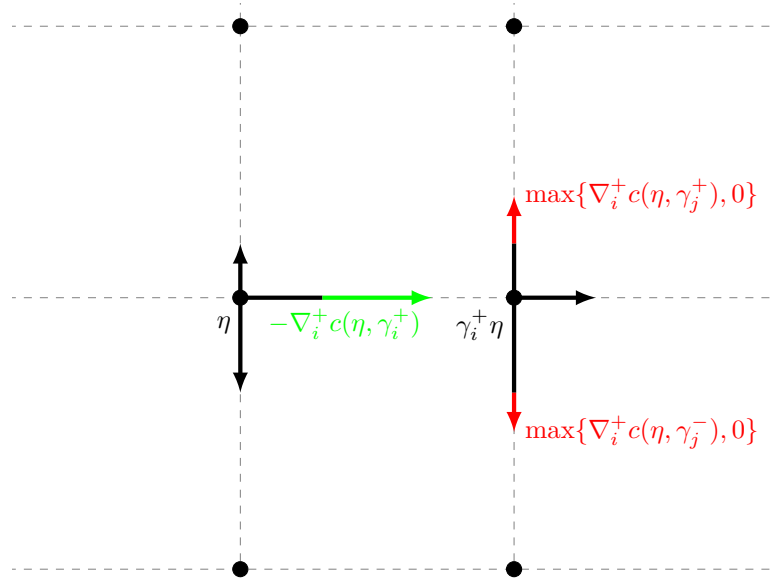
\begin{proof}[Proof of Theorem \ref{thm: LSI int random walks}]
It is straightforward to check that \ref{H3.3} implies \ref{H0} and \ref{H1} with $(\gamma^{\pm}_{i})^{-1}=(\gamma^{\mp}_{i})$ for all $1\leq i \leq d$. The proof strategy consists of concatenating Proposition \ref{cor: sufficient cond} with Corollary \ref{cor: countable state space}. For this reason we define $\Omega^n=\{ \eta\in\N^d: \eta_i\leq n\,\forall i\leq d \}$ and consider the Markov chain on $\Omega^n$ whose generator $\cL^n$ is given by \eqref{def: loc gen} and \eqref{def: loc moves loc rates}, i.e. 
\bes
\cL^n f (\eta) = \sum_{i=1}^d c(\eta,\gamma^{+,n}_i) \nabla_{i}^{+,n} f(\eta) +c(\eta,\gamma^{-,n}_i)\nabla_{i}^{-,n} f(\eta)
\ees
where for all $\eta\in \N^d$ and $i,j\in\{1,\ldots,d\}$
\bes
\gamma^{+,n}_i\eta=\eta+ \mathbf{e}_i\mathbf{1}_{\eta_i<n},\quad \gamma_i^-\eta= \eta-\mathbf{e}_i\mathbf{1}_{\eta_i>0}
\ees
and 
\bes
c(\eta,\gamma^{+,n}_i) = c(\eta,\gamma^+_i), \quad c(\eta,\gamma^{-,n}_i) = c(\eta,\gamma_i^-).
\ees

For the sake of simplicity, and since there is no ambiguity, we will keep writing $\gamma^{\pm}_i$ instead of $\gamma^{\pm,n}_i$ and adopt the same convention for discrete gradients. Likewise, we shall write $c^n(\eta,\gamma^{\pm}_i)$ instead of $c(\eta,\gamma^{\pm,n}_i)$. Remark that under the current hypothesis

\[ \cS= \{ (\eta,\gamma_i^+): \eta \in\Omega^{n}, 1\leq i \leq d\} \cup \{ (\eta,\gamma_i^-): \eta\in \Omega^{n}, \eta_i>0,\, 1\leq i \leq d\}.\]

For $(\eta,\gamma_{i}^+)\in\cS$ we define (see also figure \ref{figure: coupling support})

\be\label{def: coup rates int rw}
\couprate(\eta,\gamma^+_i\eta,\gamma,\bgamma) =
\begin{cases}
\min\{ c^n(\eta,\gamma),c^n(\gamma^+_i\eta,\gamma)\}, & \quad \mbox{if $\gamma=\bgamma\in G$,}\\
\max\{\nabla_{i}^+c^n(\eta,\bgamma),0\}, & \quad \mbox{if $\gamma=\gamma_i^+$ and $\bgamma\in G$, $\bgamma\neq\gamma_i^+,\gamma_i^-$,}\\
\max\{-\nabla_{i}^+c^n(\eta,\gamma),0\}, & \quad \mbox{if $\gamma\in G$,$\gamma\neq\gamma_i^+,\gamma_i^-$ and $\bgamma=\gamma_i^-$,}\\
\kappa^{+}(\eta,i),& \quad \mbox{if $\gamma=\gamma_i^+,\bgamma=\zero$,}\\
\kappa^{-}(\eta,i),& \quad \mbox{if $\gamma=\zero,\bgamma=\gamma^{-}_i$,}\\
0,& \quad \mbox{otherwise.}\\
\end{cases}
\ee
 Since $(\eta,\gamma^-_i)\in\cS\Rightarrow(\gamma_i^-\eta,\gamma^+_i)\in\cS$ we can also define
 \be\label{def: coup rates int rw 2} \couprate(\eta,\gamma^-_{i}\eta,\gamma,\bgamma)=\couprate(\gamma^-_{i}\eta,\gamma^{+}_i(\gamma^-_{i}\eta),\bgamma,\gamma) ,\ee
uniformly on $\gamma,\bgamma\in G^*$. A direct calculation  using \ref{H3.3} shows that \eqref{def: coup rates int rw} and \eqref{def: coup rates int rw 2} define admissible coupling rates. We now move on to prove that \eqref{eq: discrete gradient estimate} holds with $\kappa'=\kappa$. We have to show that for all $f>0$

\begin{align}\label{eq: G+G-}
\nonumber \frac12\sum_{(\eta,\gamma_i^+)\in\supp}\sum_{(\gamma,\bgamma)\in G^*}c^n(\eta,\gamma_i^+)\couprate(\eta,\gamma_i^+\eta,\gamma,\bgamma) \coupgrad_{\gamma,\bgamma} \fphi (\eta,\gamma_i^+\eta)\, \inv^n(\eta) \\
+\frac12\sum_{(\eta,\gamma_i^-)\in\supp}\sum_{(\gamma,\bgamma)\in G^*}c^n(\eta,\gamma_i^-)\couprate(\gamma_i^-\eta,\eta,\bgamma,\gamma) \coupgrad_{\gamma,\bgamma} \fphi (\eta,\gamma_i^-\eta)\, \inv^n(\eta)
 \end{align}
is bounded above by $-\kappa \cE(\phi'(f),f)$. Using \eqref{eq rev} and the fact that $\fphi$ is symmetric in its arguments, we get that the second summand equals the first one in the above expression, which can then be rewritten as $A+B+C+D$ with 

\begin{align*}
  A&=\sum_{\substack{\eta\in\Omega^n,  i\leq d\\ \gamma\in G}}c^n(\eta,\gamma_i^+)\min\{ c^n(\eta,\gamma),c^n(\gamma^+_i\eta,\gamma)\}\coupgrad_{\gamma,\gamma}\fphi(\eta,\gamma_i^+\eta)\, \inv^n(\eta), \\
  B&=\sum_{\substack{\eta\in\Omega^n \\ i\leq d}}\sum_{\substack{ \gamma\in G \\\gamma\neq \gamma_i^+,\gamma_i^-}}c^n(\eta,\gamma_i^+)\max\{\nabla_{i}^+c^n(\eta,\gamma),0\} \coupgrad_{\gamma^+_i,\gamma}\fphi(\eta,\gamma_i^+\eta)\, \inv^n(\eta),\\
   C&=\sum_{\substack{\eta\in\Omega^n \\ i\leq d}}\sum_{\substack{ \gamma\in G \\\gamma\neq \gamma_i^+,\gamma_i^-}}c^n(\eta,\gamma_i^+)\max\{-\nabla_{i}^+c^n(\eta,\gamma),0\} \coupgrad_{\gamma,\gamma_i^-}\fphi(\eta,\gamma_i^+\eta)\, \inv^n(\eta),\\
   D&=\sum_{\substack{\eta\in\Omega^n \\ i\leq d}}c^n(\eta,\gamma_i^+)[\kappa^+(\eta,i) \coupgrad_{\gamma_i^+,\zero}\fphi(\eta,\gamma_i^+\eta)+ \kappa^-(\eta,i) \coupgrad_{\zero,\gamma_i^-}\fphi(\eta,\gamma_i^+\eta) ]\, \inv^n(\eta)
\end{align*}
We first rewrite the term $A$ as $A.1+A.2+A.3$, where 

\begin{align*}
A.1 &=\sum_{\substack{\eta\in\Omega^n \\  i,j\leq d}}c^n(\eta,\gamma_i^+)c^n(\gamma^+_i\eta,\gamma^+_j)\coupgrad_{\gamma^+_j,\gamma^+_j}\fphi(\eta,\gamma_i^+\eta)\, \inv^n(\eta)\\
&+\sum_{\substack{\eta\in\Omega^n \\  i,j\leq d}} c^n(\eta,\gamma_i^+)c^n(\eta,\gamma^-_j)\coupgrad_{\gamma_j^-,\gamma_j^-}\fphi(\eta,\gamma_i^+\eta)\, \inv^n(\eta).
\end{align*}

\begin{align*}
    A.2 &=-\sum_{\substack{\eta\in\Omega^n \\  i,j\leq d}}c^n(\eta,\gamma_i^+)\max\{\nabla_{i}^+c^n(\eta,\gamma^+_j),0\}\coupgrad_{\gamma^+_j,\gamma^+_j}\fphi(\eta,\gamma_i^+\eta)\, \inv^n(\eta),\\
    A.3&=-\sum_{\substack{\eta\in\Omega^n \\  i,j\leq d}} c^n(\eta,\gamma_i^+)\max\{-\nabla_i^+c^n(\eta,\gamma^-_j),0\}\coupgrad_{\gamma_j^-,\gamma_j^-}\fphi(\eta,\gamma_i^+\eta)\, \inv^n(\eta).
\end{align*}

We now consider the terms $B$ and $C$ and observe that $B=B.1+B.2$, $C=C.1+C.2$ with
\begin{align*}
  B.1&=  \sum_{\substack{\eta\in\Omega^n \\ i\leq d}}\sum_{\substack{j\leq d \\j \neq i}}c^n(\eta,\gamma_i^+)\max\{\nabla_{i}^+c^n(\eta,\gamma_j^+),0\} \coupgrad_{\gamma^+_i,\gamma_j^+}\fphi(\eta,\gamma_i^+\eta)\, \inv^n(\eta),\\
  B.2&=  \sum_{\substack{\eta\in\Omega^n \\ i\leq d}}\sum_{\substack{j\leq d \\j \neq i}}c^n(\eta,\gamma_i^+)\max\{\nabla_{i}^+c^n(\eta,\gamma_j^-),0\} \coupgrad_{\gamma^+_i,\gamma_j^-}\fphi(\eta,\gamma_i^+\eta)\, \inv^n(\eta).
\end{align*}
and
\begin{align*}
   C.1&= \sum_{\substack{\eta\in\Omega^n \\ i\leq d}}\sum_{\substack{j\leq d \\j \neq i}}c^n(\eta,\gamma_i^+)\max\{-\nabla_{i}^+c^n(\eta,\gamma_j^-),0\} \coupgrad_{\gamma_j^-,\gamma_i^-}\fphi(\eta,\gamma_i^+\eta)\, \inv^n(\eta),\\
  C.2&=  \sum_{\substack{\eta\in\Omega^n \\ i\leq d}}\sum_{\substack{j\leq d \\j \neq i}}c^n(\eta,\gamma_i^+)\max\{-\nabla_{i}^+c^n(\eta,\gamma^+_j),0\} \coupgrad_{\gamma_j^+,\gamma_i^-}\fphi(\eta,\gamma_i^+\eta)\, \inv^n(\eta).\\
\end{align*}

We claim that $A.1=0$, $C.1+A.3=B.1+A.2=0$, $C.2=B.2=0$. We break the proof of these identities into three different steps. A fourth step concludes the proof of Theorem \ref{thm: LSI int random walks}. 
\bei
\item\underline{\emph{Step 1: \,$A.1=0.$}}
This is done using reversibility \eqref{eq rev} on the second summand of $A.1$ with
\[ F(\eta,\sigma)= \mathbf{1}_{ G^-}(\sigma)\sum_{\substack{\eta\in\Omega^n \\  i\leq d}}c^n(\eta,\gamma_i^+)\coupgrad_{\sigma,\sigma}\fphi(\eta,\gamma_i^+\eta), \quad G^-=\{\gamma_j^-:1\leq j \leq d \}. \]
and observing that that \eqref{rates: int random walks} implies $c^n(\eta,\gamma_j^+)c^n(\gamma_j^+\eta,\gamma_i^+)=c^n(\eta,\gamma_i^+)c^n(\gamma_i^+\eta,\gamma_j^+)$ for all $\eta,j,i$.
\item\underline{\emph{Step 2: \,$C.2=B.2=0.$}}
Using reversibility \eqref{eq rev} with
\bes
F(\eta,\sigma)= \mathbf{1}_{ G^+}(\sigma)\sum_{\substack{\eta\in\Omega^n \\  \gamma\in G^-,\gamma\neq \sigma,\sigma^{-1}}}\max\{\nabla_{\sigma} c^n(\eta,\gamma),0\}\coupgrad_{\sigma,\gamma}\fphi(\eta,\sigma\eta), \quad G^+=\{\gamma_i^+:1\leq i \leq d \}
\ees
and the symmetry of $f$ we obtain
\bes
B.2=\sum_{\substack{\eta\in\Omega^n \\ i\leq d}}\sum_{\substack{j\leq d \\j \neq i}}c^n(\eta,\gamma_i^-)\max\{-\nabla_{i}^-c^n(\eta,\gamma_j^-),0\} [f(\eta,\gamma_j^-\eta)-f(\eta,\gamma_i^-\eta)]\, \inv^n(\eta).
\ees
The desired conclusion is obtained rewriting the above expression in a symmetric way by exchanging the lables $i$ and $j$ and observing that \eqref{rates: int random walks double pot} implies 
\be\label{eq: potential commutation}
c^n(\eta,\gamma_i^-)\nabla_i^-c^n(\eta,\gamma^-_j)=c^n(\eta,\gamma_j^-)\nabla_j^-c^n(\eta,\gamma^-_i).
\ee
The proof that $C.2=0$ is almost identical but simpler as we do not need to invoke \eqref{eq rev}. Indeed, exchanging the labels $i$ and $j$ we arrive at
\bes
C.2=\sum_{\substack{\eta\in\Omega^n \\ i\leq d}}\sum_{\substack{j\leq d \\j \neq i}}c^n(\eta,\gamma_i^+)\max\{-\nabla_{i}^+c^n(\eta,\gamma^+_j),0\} [f(\eta,\gamma_j^+\eta)-f(\eta,\gamma_i^+\eta)]\, \inv^n(\eta).
\ees
To conclude, we observe that \eqref{rates: int random walks double pot} implies
\be\label{eq: potential commutation +}
c^n(\eta,\gamma_i^+)\nabla_i^+c^n(\eta,\gamma^+_j)=c^n(\eta,\gamma_j^+)\nabla_j^+c^n(\eta,\gamma^+_i).
\ee

\item\underline{\emph{Step 3:\, $C.1+A.3=B.1+A.2=0$}}
We observe that, since \ref{H3.3} implies $\nabla_{i}^+c(\eta,\gamma_i^-)\geq 0$: 
\begin{align*}
C.1+A.3= \sum_{\substack{\eta\in\Omega^n \\  i,j\leq d}} c^n(\eta,\gamma_i^+)\max\{-\nabla_i^+c^n(\eta,\gamma^-_j),0\}\,[f(\eta,\gamma_j^-\eta)- f(\gamma_j^-\eta,\gamma_j^-\gamma^+_{i}\eta) ] \inv^n(\eta),
\end{align*}
Using reversibility \eqref{eq rev} and the symmetry of $\fphi$ we obtain the equivalent expression
\bes
\sum_{\substack{\eta\in\Omega^n \\  i,j\leq d}} c^n(\eta,\gamma_i^-)\max\{\nabla_i^-c^n(\eta,\gamma^-_j),0\}\,[f(\gamma_i^-\eta,\gamma_j^-\gamma_i^-\eta)- f(\gamma_j^-\eta,\gamma_j^-\gamma_i^-\eta) ] \inv^n(\eta).
\ees

Arguing as in \emph{Step 2}, i.e. rewriting the above expression in a symmetric way exchanging the labels $j$ and $i$ and using \eqref{eq: potential commutation} we get the desired result. As in the previous step, the proof that $B.1+A.2=0$ is almost identical, the only differences being that we do not need to invoke\eqref{eq rev} and \eqref{eq: potential commutation} is replaced by \eqref{eq: potential commutation +}.

\item\underline{\emph{Step 4:\, conclusion}}
Combining \ref{H3.3} with the fact that $\Phi$ vanishes on the diagonal we obtain
\bes
D \leq - \kappa\sum_{\substack{\eta\in\Omega^n \\  i\leq d}} c^n(\eta,\gamma^{+}_i) \Phi(f(\eta),f(\gamma^{+}_i\eta)) \inv^n(\eta).
\ees
Using the reversibility \eqref{eq rev} one last time, it follows that the right hand side of the last expression equals $-\kappa\cE(\phi'(f),f)$. Combining this with $A+B+C=0$, which we proved in the former steps, we conclude that \eqref{eq: discrete gradient estimate} holds with $\kappa'=\kappa$. From the construction of the coupling rates we also have that \eqref{def: kappa'''} holds with $\kappa'''=\kappa$. Applying Proposition \ref{cor: sufficient cond} for any $n$ and eventually using Corollary \ref{cor: countable state space} concludes the proof.
\eei
\end{proof}

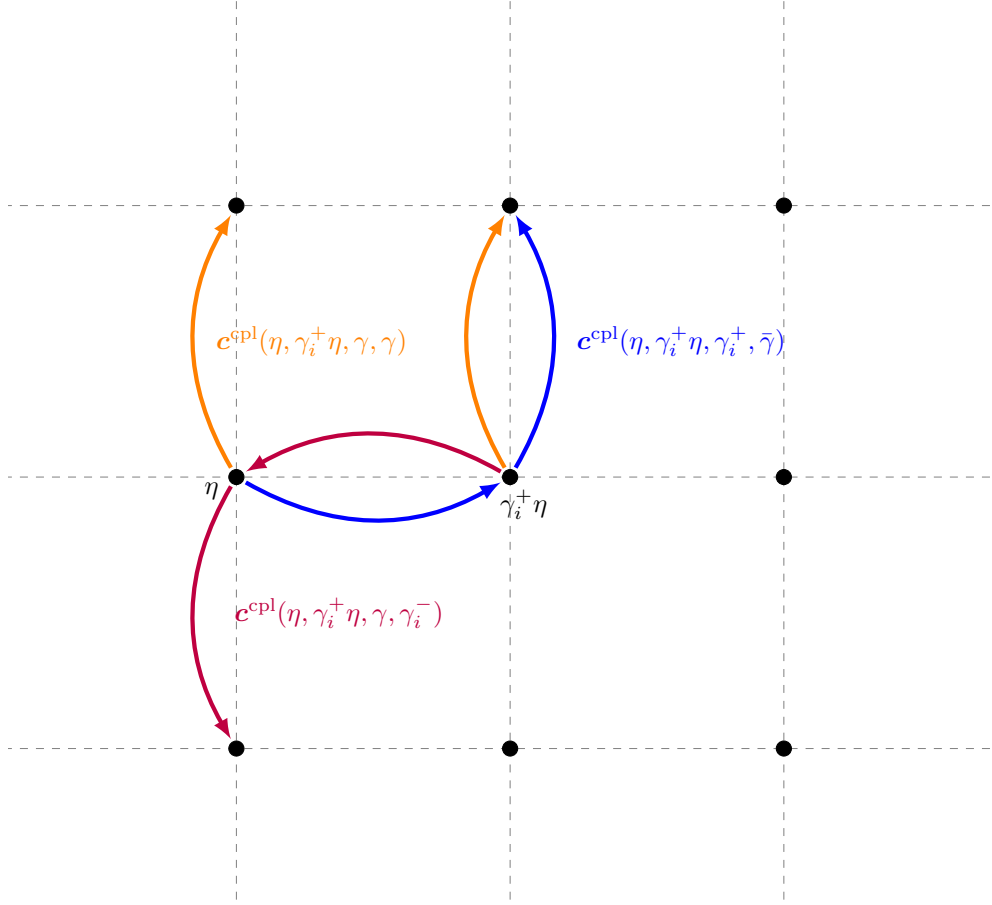
\begin{figure}[ht]
  \centering
  \begin{tikzpicture}
   \coordinate (Origin)   at (0,0);

    \clip (-3,-2) rectangle (10cm,10cm); 
    \pgftransformcm{1.8}{0}{0}{1.8}{\pgfpoint{0cm}{0cm}}
    \coordinate (eta) at (0,2);
    \coordinate (breta) at (2,2);
    \draw[style=help lines,dashed] (-14,-14) grid[step=2cm] (14,14);
    \foreach \x in {-7,-6,...,7}{
      \foreach \y in {-7,-6,...,7}{
        \node[draw,circle,inner sep=2pt,fill] at (2*\x,2*\y) {};
      }
    }
    \node at (-0.18,1.9) {$\eta$};
    \node at (2.1,1.8) {$\gamma_i^+\eta$};
    \node at (3.25,3) [blue]{$\bm{c}^{\mathrm{cpl}}(\eta,\gamma_i^+\eta,\gamma_i^+,\bgamma)$};
    \node at (0.55,3) [ orange]{$\bm{c}^{\mathrm{cpl}}(\eta,\gamma_i^+\eta,\gamma,\gamma)$};
    \node at (0.75,1) [purple]{$\bm{c}^{\mathrm{cpl}}(\eta,\gamma_i^+\eta,\gamma,\gamma_i^-)$};


  \node (1) at (0,2) {};
  \node (2) at (2,2)  {};
  \node (3) at (2,4)  {};
  \node (4) at (0,4)  {};
  \node (5) at (0,0)  {};
  \node (6) at (2,0)  {};
  \path[every node/.style={font=\sffamily\small}]
    (1) edge [bend right, blue, ultra thick, -latex] node [right] {} (2)
    (2) edge [blue, ultra thick, -latex, bend right] node [right] {} (3)
    (2) edge [orange, ultra thick, -latex, bend left] node [right] {} (3)
    (1) edge[orange, ultra thick, -latex, bend left] node [left] {} (4)
    (1) edge [purple,ultra thick, -latex, bend right] node [right] {} (5)
    (2) edge [purple, ultra thick, -latex, bend right] node [left] {} (1);
  
  \end{tikzpicture}
  \caption[fuffi]{The rates defined at \eqref{def: coup rates int rw} couple the dynamics of two random walkers $X^\eta_t,X^{\gamma_i^+\eta}_t$ starting at $\eta,\gamma_i^+\eta$ respectively in such a way that, on a short time interval $[0,\varepsilon]$ only one of the following movements can be observed:\\ \newline
  
  $\bullet$ The two walkers meet at $\eta$ (resp. $\gamma_i^+\eta$) with probability $\varepsilon \kappa^-(\eta,i)$ (resp. $\varepsilon\kappa^+(\eta,i)$).\\ \newline
    $\bullet$  The two walkers move in parallel making the same move $\gamma$ with probability $\varepsilon \bm{c}^{\mathrm{cpl}}(\eta,\gamma_i^+\eta,\gamma,\gamma)$. (yellow lines)\\ \newline
$\bullet$ The walker $X^{\gamma_i^+\eta}_t$ runs after the walker $X^{\eta}_t$ with probability  $\varepsilon \bm{c}^{\mathrm{cpl}}(\eta,\gamma_i^+\eta,\gamma,\gamma_i^-)$. (purple lines)\\ \newline
$\bullet$ The walker $X^{\eta}_t$ runs after the walker $X^{\gamma_i^+\eta}_t$ with probability  $\varepsilon \bm{c}^{\mathrm{cpl}}(\eta,\gamma_i^+\eta,\gamma_i^+,\gamma)$. (blue lines)
} 
  \label{figure: coupling support}
\end{figure}
\subsection{Examples}\label{sec: int rw examples}

In this section, we present some examples where explicit lower bounds on $\kappa_{\phi}$ can be obtained thanks to Theorem \ref{thm: LSI int random walks}. In \cite[Sec 3.2]{dai2013entropy}, the authors managed to establish MLSI for a two dimensional non perturbative example corresponding to $V^+(\eta)=h(\eta_1+\eta_2), V^{-}(\eta)=\sum_{j=1}^d\log(\eta_i!)$ with $h$ convex and increasing. There, they raised the question of how to generalize this result to a dimension $d>2$. Thanks to Theorem \ref{thm: LSI int random walks} we can answer this question in the next Corollary. In order to state this result, and in the remainder of this article, for a function $h:\R\rightarrow\R$ we use the notation $\nabla^{+}h(m)$ to indicate the increment $h(m+1)-h(m)$. Throughout this section the potential $V^-$ is fixed to be 

\bes
V^{-}(\eta)=\sum_{i=1}^d\log(\lambda)\eta_i+ \log(\eta_i!)
\ees
and we will abbreviate $V^+$ with $V$. We are therefore in the setting where $\cL$ given by \eqref{def: gen int rw},\eqref{rates: int random walks}.
\begin{cor}\label{cor: symmetric interaction}
Let $|\eta|=\sum_{i=1}^d\eta_i$, $h:\R\rightarrow\R$ be a convex function and set $V(\eta)=\beta h(|\eta|)$. Consider the generator $\cL$ given by \eqref{def: gen int rw}\eqref{rates: int random walks}. If

\be\label{hyp:symmetric interaction}
\inf_{m\in\N} \lambda-(d-1)[\exp(-\beta\nabla^{+}h(m))-\exp(-\beta\nabla^{+}h(m+1))] \geq 0, 
\ee

then the conclusion Theorem \ref{thm: LSI int random walks} holds with 

\be\label{eq: kappa symmetric interaction} 
\kappa=\inf_{m\in\N} \lambda-(d-2)[\exp(-\beta\nabla^{+}h(m))-\exp(-\beta\nabla^{+}h(m+1))].
\ee

In particular, if $h$ is strictly increasing  and 

\be\label{hyp: low temperature}
\beta \geq \frac{\log(d-1)-\log(\lambda)}{h(1)-h(0)},
\ee

then the conclusion of Theorem \ref{thm: LSI int random walks} holds with

\be\label{eq: kappa symmetric interaction explicit}
\kappa = \lambda -(d-2) \exp(-\beta\nabla_+h(0)).
\ee

\end{cor}

Perturbative criteria typically assert that a probability measure of the form $\inv=\exp(-\beta V)\bm\mu$ where $\bm\mu$ is a product measure satisfy MLSI provided $\beta$ is small enough. This is often called a weak interaction/high temperature condition. On the contrary, \eqref{hyp: low temperature} asks for a lower bound on $\beta$ and is therefore a non perturbative condition. 

\begin{proof}
Using the convexity of $h$ and the definition of $V$ we obtain that 
\[ -\nabla_i^+c(\eta,\gamma_j^+)= [\exp(-\beta\nabla^{+}h(|\eta|))-\exp(-\beta\nabla^{+}h(|\eta|+1))]\] uniformly on $\eta,i,j$. Thus, by convexity of $h$ we obtain that $-\nabla_i^+c(\eta,\gamma_j^+)\geq0$ and that
\bes
\kappa^+(\eta,i)= -\nabla_i^+c(\eta,\gamma_i^+) \geq 0.
\ees
For the same reason
\bes 
\kappa^{-}(\eta,i)=\lambda- (d-1) \big[\exp(-\beta\nabla_+h(|\eta|))-\exp(-\beta\nabla_+h(|\eta|+1))\big],  \quad \forall \eta\in\N^d,i\leq d.
\ees

Therefore, \eqref{hyp:symmetric interaction} implies \ref{H3.3} and Theorem \ref{thm: LSI int random walks} holds with $\kappa$ given by \eqref{eq: kappa symmetric interaction}. To prove the last statement, it suffices to observe that by convexity of $h$, $\kappa^{-}(\eta,i) \geq \lambda- (d-1) \exp(-\beta \nabla^+h(0))$.
\end{proof}
 In the next proposition we show that, combining Theorem \ref{thm: LSI int random walks} with a perturbative argument we can establish MLSI with a positive constant for any $\beta>0$ and relax both the assumption that $h$ is convex increasing and the pointwise condition \eqref{hyp:symmetric interaction}. The price to pay is that we loose the precise control on the constants \eqref{eq: kappa symmetric interaction explicit}.

\begin{prop}
Let $h:\R\rightarrow\R$ be convex  outside a finite interval and set $V(\eta)=\beta h(|\eta|)$. Consider the generator $\cL$ given by \eqref{def: gen int rw} and \eqref{rates: int random walks}. Then there exists $\kappa_1>0$ such that MLSI \eqref{log sob} holds.
\end{prop}

\begin{proof}
Since $h$ is convex outside a finite interval we have that $\exp(-\beta\nabla^+h(m))$ is decreasing
for $m$ large enough and $\lim_{m\rightarrow+\infty}\exp(-\beta\nabla^+h(m))$ exists and is finite, from which it follows that $\lim_{m\rightarrow+\infty} \nabla^+\exp(-\beta\nabla^+h(m))=0$.
For $\varepsilon>0$ small enough consider $M_{\varepsilon}$ large enough such that $h$ is convex and $\nabla^+\exp(-\beta\nabla^+h(m))\geq-\frac{\varepsilon}{d-1}$ outside $[0,M_{\varepsilon}]$. Define
$\tilde{h}:\N\rightarrow \R_{\geq 0}$ as follows

\bes
\tilde{h}(m)=\begin{cases}h(m), \quad & \mbox{if $m \geq M_{\varepsilon}$}\\
h(M_\varepsilon) - (h(M_{\varepsilon}+1)-h(M_{\varepsilon})) (M_{\varepsilon}-m),\quad & \mbox{if $m \leq M_{\varepsilon}-1$}
\end{cases}
\ees
The function $\tilde{h}$ satisfies \eqref{hyp:symmetric interaction} because
\bes
\nabla^+\exp(-\beta \nabla^{+} \tilde{h}(m)) = \begin{cases}\nabla^+\exp(-\beta \nabla^{+} h(m)) \leq -\frac{\varepsilon}{d-1},  \quad & \mbox{if $m \geq M_{\varepsilon}$}\\
0,\quad & \mbox{if $m \leq M_{\varepsilon}-1$}
\end{cases}
\ees
If we define $\tilde{V}=\tilde{h}(|\eta|)$ an application of Corollary \ref{cor: symmetric interaction} yields MLSI with constant $\kappa_1=\lambda-\varepsilon(d-2)$. It is easily seen that $\inv=\frac{1}{Z}\exp(-V(\eta))\bm\mu_{\lambda}$ and $\tilde{\inv}=\frac{1}{\tilde{Z}}\exp(-\tilde{V}(\eta))\bm\mu_{\lambda}$ are equivalent probability measures, i.e. $1/K\leq\frac{\De \inv}{\De \tilde{\inv}}\leq K$ for some finite $K$. A standard perturbative argument (see e.g. \cite{holley1986logarithmic}) gives the desired conclusion.
\end{proof}

For diffusions on $\RD$ \eqref{Langevin}, it is a lower bound on the spectrum of the Hessian of $V$ viewed as a quadratic form that implies LSI. In the next Corollary, we show that for interacting random walks a \emph{pointwise} bound on the entries of the Hessian of $V$ plus a local condition \eqref{hyp: oddly convex potentials} on the behavior of $V$ at the origin imply the family of convex Sobolev inequalites. 

\begin{cor}\label{cor: oddly convex potentials}
Let $V:\RD \rightarrow \R$ be twice continuously differentiable and such that 
\be\label{hyp: pointwise bound hessian} \partial_{x_i x_j}V(x) \geq 0 \ee

holds uniformly in $x\in \RD$. Consider the generator $\cL$ given by \eqref{def: gen int rw},\eqref{rates: int random walks} and assume that

\be\label{hyp: oddly convex potentials general} 
\inf_{\eta\in \N^d,i=1,\ldots,d} \lambda-\sum_{\substack{j=1\\ j\neq i}}^d [\exp(-\nabla_j^+V(\eta))-\exp(-\nabla_j^+V(\gamma_i^+\eta))] \geq 0.
\ee
Then
\bei
\item[(i)] The conclusion of Theorem \ref{thm: LSI int random walks} holds with $\kappa$ given by 
\bes
\inf_{\eta\in \N^d,i=1,\ldots,d} \lambda + [\exp(-\nabla_i^+V(\eta))-\exp(-\nabla_i^+V(\gamma_i^+\eta))] -\sum_{\substack{j=1\\ j\neq i}}^d [\exp(-\nabla_j^+V(\eta))-\exp(-\nabla_j^+V(\gamma_i^+\eta))] 
\ees
\item[(ii)] If
\be\label{hyp: oddly convex potentials} 
\min_{i=1,\ldots,d} \lambda-\sum_{\substack{j=1\\ j\neq i}}^d \exp(-\nabla_j^+V(\mathbf{0})) \geq 0,
\ee
where $\mathbf{0}=(0,\ldots,0)\in\N^d$, then the conclusion of Theorem \ref{thm: LSI int random walks} holds with $\kappa$ given by \eqref{hyp: oddly convex potentials}.
\eei
\end{cor}

\begin{proof}
We deduce from \eqref{hyp: pointwise bound hessian} that 

\be\label{eq: oddly convex discrete potential}\nabla_i^+\nabla_j^+V(\eta)\geq 0, \quad   \forall \eta\in \N^d,i,j\leq d.\ee

But then, using the definition of the transition rates \eqref{rates: int random walks} we obtain $\nabla_i^+c(\eta,\gamma_j^+)\leq 0$, which gives 
\[\kappa^{+}(\eta,i) = -\nabla_i^+c(\eta,\gamma_i^+)=\exp(-V^{+}(\eta))-\exp(-V^+(\gamma_i^+\eta))\geq 0, \quad \forall \eta\in \N^d,i\leq d\] 
and 
\bes 
\kappa^-(\eta,i) = \lambda-\sum_{\substack{j=1\\ j\neq i}}^d \exp(-\nabla_j^{+} V(\eta))-\exp(-\nabla_j^{+} V(\gamma^+_{i}\eta)), \quad \forall \eta\in\N^d,i\leq d,
\ees
Therefore if \eqref{hyp: oddly convex potentials general} holds, so does \ref{H3.3}. An application of Theorem \ref{thm: LSI int random walks} concludes the proof of $(i)$. From \eqref{eq: oddly convex discrete potential} we obtain 

\bes
\nabla_j^{+}V(\eta) \geq \nabla_j^{+}V(\mathbf{0}), \quad \forall \eta\in\N^d,j\leq d,
\ees

that yields the bound

\[\forall \eta,i \quad \kappa^-(\eta,i) \geq \lambda-\sum_{\substack{j=1\\ j\neq i}}^d \exp(-\nabla_j^{+} V(\mathbf{0})),\] 
This last bound together with $(i)$ prove $(ii)$.

\end{proof}

We did not use introduce an "inverse temperature" parameter $\beta$ in Corollary \ref{cor: oddly convex potentials}. If we had done so, we could have seen that, as for Corollary \ref{cor: symmetric interaction}, the local condition \eqref{hyp: oddly convex potentials} is always satisfied in the low temperature regime $\beta\rightarrow+\infty$.

\subsection{Contraction of the Wasserstein distance}

Let $p\geq1$ and $d(\eta,\breta)$ be the graph distance on $\N^d$:

\bes
d(\eta,\breta)=\sum_{i=1}^d |\eta_i-\breta_i|.
\ees
For given $\mu,\nu\in \cP_p(\N^d)$ with finite $p$-th moment the Wasserstein distance of order $p$ is defined as

\bes
W_p(\mu,\nu) = \Big(\inf_{\pi\in\Pi(\mu,\nu)}\sum_{\eta',\eta''} d^p(\eta',\eta'') \pi(\eta',\eta'')\Big)^{1/p},
\ees

where $\Pi(\mu,\nu)$ is, as usual, the set of all couplings of $\mu$ and $\nu$. In the next Theorem we show that the hypothesis needed for Theorem \ref{thm: LSI int random walks} are equivalent to certain contractivity properties of the Wasserstein distance along the semigroup $(S_t)_{t\geq0}$ generated by $\cL$. In the statement of the next theorem, for $\mu,\nu\in\cP_p(\N^d)$ we denote by $\mu_t,\nu_t$ the laws $\mu S_t,\nu S_t$, i.e. $\mu_t$ (resp. $\nu_t$) is the distribution at time $t$ of a Markov chain with generator $\cL$ and initial distribution $\mu$ (resp.$\nu$). 
\begin{theorem}\label{thm: wass contr}
Let $V^-,V^+:\N^d\rightarrow\R$ be given and $\cL$ be defined by \eqref{def: gen int rw double pot}\eqref{rates: int random walks double pot}. Moreover, assume that \ref{H0}\ref{H1} hold and that

\be\label{hyp: non explosion} \inf_{\substack{\eta\in\N^d \\ i \leq d }} \nabla_i^+V^-(\eta)>-\infty, \quad \inf\{\nabla_i^-V^-(\eta): \eta_i\leq K\} >-\infty \quad \forall K>0,i\leq d.\ee 

The following statements are equivalent for any $\kappa>0$

\bei
\item[(i)] The estimate  

\be\label{Wass contr random walks}
 W_p(\mu_t,\nu_t)  \leq \exp\big(-\frac{\kappa}{p} t\big) W_p(\mu,\nu)
\ee

holds uniformly on $\mu,\nu \in \cP_p(\N^d)$, $t>0$ and $p\geq1$.

\item[(ii)]  $V^-,V^+$ satisfy  \ref{H3.3}   and
\be\label{eq: coup curvature} \inf_{\substack{\eta\in\N^d\\ i\leq d}} \kappa^+(\eta,i)+\kappa^-(\eta,i) \geq \kappa, \ee
where $\kappa^{+}(\eta,i)$ and $\kappa^{-}(\eta,i)$ are defined at \eqref{def: kappa+} and \eqref{def: kappa-}.
\eei
\end{theorem}

We have the following Corollary.

\begin{cor}\label{cor: contr implies log sob}
Let $V^-,V^+:\N^d\longrightarrow\R$ be given and $\cL$ be defined by \eqref{def: gen int rw double pot}\eqref{rates: int random walks double pot}. Assume that \ref{H0},\ref{H1} and \eqref{hyp: non explosion} hold. If the contraction estimate \eqref{Wass contr random walks}
holds uniformly on $\mu,\nu \in \cP_p(\N^d)$, $t>0$ and $p\geq1$, then for all $\phi$ satisfying \ref{H2} the convex Sobolev inequality holds with constant $\kappa_{\phi}=\kappa$ and for any $\alpha\in[1,2]$ the Beckner inequality holds with $\kappa_{\alpha}=\alpha\kappa$.
\end{cor}

Note that if $\cL$ satisfies \ref{H0},\ref{H1} the only reversible measure is $\inv=\frac{1}{Z}\exp(-V^+-V^-)$.
The technical assumption \eqref{hyp: non explosion} is made to simplify the proof and put forward the main ideas but does not play an essential role. It is likely that this assumption can be largely weakened. Moreover, observe that the second condition therein is always satisfied in the setting \eqref{rates: int random walks}, i.e. when $V^{-}(\eta)=\sum_{i=1}^d\log(\lambda) \eta_i+ \log(\eta_i!)$.

\begin{remark} 
For the Langevin dynamics \eqref{Langevin} even more is known \cite{von2005transport}: indeed, Wasserstein contraction for $p=1$ implies, and is in fact equivalent to the same property for $p\geq 1$, all this properties being equivalent to the $\kappa$ convexity of $V$. In particular, this implies that Wasserstein contraction for $p=1$ suffices to conclude that the logarithmic Sobolev inequality holds. At the moment of writing, we do not know if the same result holds in $\N^d$. However, it is not hard to see that for continuous Markov chains on $\N^d$, there is no equivalence between Wasserstein contraction for different values of $p$.
\end{remark}

In the proof of the Theorem, we will need the following technical lemma, whose proof we defer to the appendix

\begin{lemma}\label{lem: 1 jump approx}
Under the same hypothesis of Theorem \ref{thm: wass contr} consider $\eta\in\N^d$ and set $\mu=\delta_{\eta}\in\cP_p(\N^d)$. Then there exist $C>0$ such that
\be\label{eq: 1 jump approx}
W^p_p(\mu_t,\bar{\mu_t}) \leq C t^{2}
\ee
holds for $t$ small enough, where 

\be\label{def: barmu}
\bar{\mu}_t = \big(1-t\sum_{\gamma\in G}c(\eta,\gamma)\big)\delta_{\eta}+ t\sum_{\gamma\in G}c(\eta,\gamma)\delta_{\gamma\eta} 
\ee
\end{lemma}

\begin{proof}[Proof of Theorem \ref{thm: wass contr}]
We first prove (i)$\Rightarrow$(ii) in two steps. In the first step we show that $(i)$ implies \ref{H3.3}, whereas in the second step we show that it implies \eqref{eq: coup curvature}.

\bei \item\underline{Step1: (i)$\Rightarrow$ \ref{H3.3}  } If \ref{H3.3}   is violated, then there exist $\eta\in\N^d$ and $i\leq d$,such that $\kappa^{-}(\eta,i)<0$ or $\kappa^{+}(\eta,i)<0$. We only treat the case $\kappa^{-}(\eta,i)<0$. The proof for the other case is almost identical and we omit the details. We set $\mu = \delta_{\eta}$, $\nu=\delta_{\gamma^{+}_i\eta}$ and define $\bar\mu_t,\bar\nu_t$ as in \eqref{def: barmu}. Invoking Lemma \ref{lem: 1 jump approx} we have that 
\be\label{eq: 1 jump wass approx} W^p_p(\mu_t,\nu_t)=W^p_p(\bar\mu_t,\bar\nu_t)+ o(t),\ee
where as usual by $o(t)$ we denote a function such that $o(t)/t\rightarrow0$ as $t\rightarrow0$.
Consider an arbitrary coupling $\pi$ of $\bar\mu_t$ and $\bar\nu_t$. Then $\pi$ is supported on the set
$\{ (\gamma\eta,\bgamma\gamma^+_i\eta): \gamma,\bgamma\in G^* \}$. We claim that $\pi[d(\eta',\eta'')\geq 2]\geq-t\kappa^-(\eta,i)$. To prove this, we observe that
\bes
\pi[d(\eta',\eta'')\geq 2] \geq \pi[\eta'=\gamma^{-}_i\eta,\eta''\neq\eta] + \sum_{\substack{\gamma\in G\\\gamma\neq\gamma^{+}_i,\gamma_i^-}} \pi[\eta'=\gamma\eta,\eta''\notin\{ \eta,\gamma\gamma^{+}_i\eta\} ].
\ees

We have
 \begin{align*}
 &\pi[\eta'=\gamma^{-}_i\eta,\eta''\neq\eta] +\sum_{\substack{\gamma\in G\\\gamma\neq\gamma^{+}_i,\gamma_i^-}} \pi[\eta'=\gamma\eta,\eta''\notin\{ \eta,\gamma\gamma^{+}_i\eta\} ]\\
 =&\pi[\eta'=\gamma^{-}_i\eta]-\pi[(\gamma^-_i\eta,\eta)]+ \sum_{\substack{\gamma\in G\\\gamma\neq\gamma^{+}_i,\gamma_i^-}} \pi[\eta'=\gamma\eta]-\pi[(\gamma\eta,\eta)]-\pi[(\gamma\eta,\gamma\gamma^{+}_i\eta) ]\\
 =&tc(\eta,\gamma^{-}_i) - \sum_{\substack{\gamma\in G\\\gamma\neq\gamma^{+}_i,\gamma_i^-}}\pi[(\gamma\eta,\eta)]+\sum_{\substack{\gamma\in G\\\gamma\neq\gamma^{+}_i,\gamma_i^-}} \pi[\eta'=\gamma\eta]-\pi[(\gamma\eta,\gamma\gamma^{+}_i\eta)]\\
 \geq &tc(\eta,\gamma^{-}_i) -\pi[\eta''=\eta]+\sum_{\substack{\gamma\in G\\\gamma\neq\gamma^{+}_i,\gamma_i^-}}\max\{\pi[\eta'=\gamma\eta]-\pi[\eta''=\gamma\gamma_i^+\eta],0\}\\
 \geq& t[c(\eta,\gamma^{-}_i) -c(\gamma^+_i\eta,\gamma^-_i)] +t\sum_{\substack{\gamma\in G\\\gamma\neq\gamma^{+}_i,\gamma_i^-}} \max\{-\nabla_i^+c(\eta,\gamma),0\}\\
 =&-t\kappa^-(\eta,i),
\end{align*}
which proves the claim. We also observe that $\pi[d(\eta',\eta'')=0]=\pi[(\eta,\eta)]+\pi[(\gamma^{+}_i \eta,\gamma^{+}_i\eta)]\leq t [c(\gamma^+_i\eta,\gamma^{-}_i)+c(\eta,\gamma^{+}_i)]$. Therefore, if $\bar\pi_t$ is the optimal coupling for $W_p(\bar\mu_t,\bar\nu_t)$ we have
\begin{align*}
    \frac{W^p_{p}(\bar\mu_t,\bar\nu_t)-W^p_{p}(\mu,\nu)}{t} &\geq \frac{(2^{p}-1)\bar\pi_t[d(\eta',\eta'')\geq 2]-\bar\pi_t[d(\eta',\eta'')=0]}{t}\\
    &\geq -(2^p-1)\kappa^-(\eta,i)-[c(\eta,\gamma^{+}_i)+c(\gamma^{+}_i\eta,\gamma^-_i)]
\end{align*}
Since this quantity is strictly positive  for $p$ large enough we obtain a contradiction with \eqref{Wass contr random walks} after letting $t\rightarrow 0$ and recalling that $W_p(\mu_t,\nu_t)=W_p(\bar\mu_t,\bar\nu_t)+o(t)$. The proof that $(i)$ implies \ref{H3.3}   is now complete. 
\item \underline{Step 2: (i)$\Rightarrow$ \eqref{eq: coup curvature}.} Consider an arbitrary pair $\eta,i$, set $\mu=\delta_{\eta}$,$\nu=\delta_{\gamma^+_i\eta}$ and let $\bar\mu_t,\bar\nu_t$ as in \eqref{def: barmu}. To do the proof, we construct explicitly the optimal coupling $\bar\pi_t$ for $W_p(\bar\mu_t,\bar\nu_t)$, which is given by setting $\bar\pi_t[(\eta',\eta'')]=0$ outside the set $\{(\gamma\eta,\bgamma\gamma_i^+\eta): \gamma,\bgamma\in G^* \}$ and defining
\be\label{eq: opt coupl}
\bar{\pi}_t(\gamma\eta,\bgamma\gamma_i^+\eta) = \begin{cases} 
t\min\{c(\eta,\gamma),c(\gamma_i^+\eta,\gamma) \}, & \quad \mbox{if $\gamma=\bgamma\in G$},\\
t\max\{\nabla_i^+c(\eta,\bgamma),0 \}, & \quad \mbox{if $\gamma=\gamma_i^+$ and $\bgamma\in G,\bgamma\neq\gamma^{+}_i,\gamma_i^-$},\\
t\max\{-\nabla_i^+c(\eta,\gamma),0 \}, & \quad \mbox{if $\gamma\in G,\gamma\neq\gamma^{+}_i,\gamma_i^-$ and $\bgamma=\gamma_i^-$},\\
t\kappa^+(\eta,i), & \quad \mbox{if $\gamma=\gamma_i^+,\bgamma=\zero$},\\
t\kappa^-(\eta,i), & \quad \mbox{if $\gamma=\zero,\bgamma=\gamma^-_i$},\\
\bar\mu_{t}(\eta)-t\kappa^-(\eta,i) (=\bar\nu_{t}(\eta)-t\kappa^+(\eta,i)),& \quad \mbox{if $\gamma=\zero,\bgamma=\zero$.}
\end{cases}
\ee

 The admissibility and optimality of $\bar\pi_t$ are shown at Lemma \ref{optimalcoupling}, which we prove separately. By construction, if  $(\eta',\eta'')$ is in the support of $\bar\pi_t$ then $d(\eta',\eta'')$ is worth $0$ if $(\eta',\eta'')=(\eta,\eta),(\gamma^+_i\eta,\gamma^+_i\eta)$ and $1$ otherwise. This gives
\be\label{eq: wass contr kappa+}
W^p_{p}(\bar\mu_t,\bar\nu_t)= 1-t[\kappa^-(\eta,i)+\kappa^+(\eta,i)]
\ee
 Therefore, invoking \eqref{eq: 1 jump wass approx} and recalling that $W_p(\mu,\nu)=1$ we obtain
 \be\label{eq: wass contr kappa+ 2}
 \frac{\De}{\De t} W^p_{p}(\mu_t,\nu_t) = -[\kappa^-(\eta,i)+\kappa^+(\eta,i)] \quad \text{and} \quad \frac{\De}{\De t} W_{p}(\mu_t,\nu_t) =-\frac1p[\kappa^-(\eta,i)+\kappa^+(\eta,i)].
 \ee
Comparing this with the hypothesis \eqref{Wass contr random walks} yields \eqref{eq: coup curvature} since the choice of $\eta$ and $i$ was arbitrary.
\eei
 \underline{Proof of (ii)$\Rightarrow $(i)}
Let $p\geq1,\eta,i$ be fixed, $\mu=\delta_{\eta},\nu=\delta_{\gamma_i^+\eta}$ and $\bar\mu_t,\bar\nu_t$ be defined as before via \eqref{def: barmu}. The hypothesis (ii) implies that the coupling $\bar\pi_t$ defined at \eqref{eq: opt coupl} is admissible and from Lemma \ref{optimalcoupling} we get that $\bar\pi_t$ is optimal for $W_p(\bar{\mu}_t,\bar{\nu}_t)$. Following the proof of Step 2, we obtain the relation \eqref{eq: wass contr kappa+ 2}. Using \eqref{eq: coup curvature} in there we arrive at 
\be\label{eq: diff wass contr}
\frac{\De}{\De t} W_{p}(\mu_t,\nu_t)\Big|_{t=0} \leq -\frac{\kappa}{p} W_{p}(\mu,\nu)
\ee
 We can extend \eqref{eq: diff wass contr} to any pair Dirac measures $\mu=\delta_{\eta},\nu=\delta_{\breta}$ by using the bound \eqref{eq: diff wass contr} along a geodesic connecting $\eta$ and $\breta$ and the triangular inequality. Next, consider two arbitrarily compactly supported $\mu$ and $\nu$ in $\cP_p(\N^d)$ and denote by $\pi_0$ the optimal couplingfor $W_p(\mu,\nu)$. For any $(\eta,\breta)\in \mathrm{supp}(\pi_0)$, let $\pi_t^{\eta,\breta}$ the optimal coupling for $W_p(\mu^{\eta}_t,\nu^{\eta}_t)$, where $\mu^{\eta}_0=\delta_{\eta},\nu^{\breta}_0=\delta_{\breta}.$ By constructing $\pi_t\in \Pi(\mu_t,\nu_t)$ as follows
\bes
\pi_t(\eta',\eta'')= \sum_{\eta,\breta \in \N^d}\pi_0(\eta,\breta)\pi_t^{\eta,\breta}(\eta',\eta'')
\ees
it is easily seen that \eqref{eq: diff wass contr} holds. The extension to non compact probability measures in $\cP_p(\RD)$ follows a standard approximation argument. The differential inequality \eqref{eq: diff wass contr} extends to an arbitrary value of $t$ by Markovianity. An application of Gr\"onwall'lemma concludes the proof.
\end{proof}
Here we prove the auxiliary Lemma needed for the proof of Theorem \ref{thm: wass contr}. To follow the proof, it may be helpful to refer to Figure \ref{figure: coupling support}.
\begin{lemma}\label{optimalcoupling}
Let the hypothesis of Theorem \ref{thm: wass contr} hold and \ref{H3.3} hold as well. For $\mu=\delta_{\eta},\nu=\delta_{\gamma^+_i\eta}$ and $t$ small enough, let $\bar\mu_t,\bar\nu_t$ be given by \eqref{def: barmu}. Next, define $\bar\pi_t\in\cP(\N^d\times\N^d)$  by setting 

\bes
\bar{\pi}_t(\gamma\eta,\bgamma\gamma_i^+\eta) = \begin{cases} 
t\min\{c(\eta,\gamma),c(\gamma_i^+\eta,\gamma) \}, & \quad \mbox{if $\gamma=\bgamma\in G$},\\
t\max\{\nabla_i^+c(\eta,\bgamma),0 \}, & \quad \mbox{if $\gamma=\gamma_i^+$ and $\bgamma\in G,\bgamma\neq\gamma^{+}_i,\gamma_i^-$},\\
t\max\{-\nabla_i^+c(\eta,\gamma),0 \}, & \quad \mbox{if $\gamma\in G,\gamma\neq\gamma^{+}_i,\gamma_i^-$ and $\bgamma=\gamma_i^-$},\\
t\kappa^+(\eta,i), & \quad \mbox{if $\gamma=\gamma_i^+,\bgamma=\zero$},\\
t\kappa^-(\eta,i), & \quad \mbox{if $\gamma=\zero,\bgamma=\gamma^-_i$},\\
\bar\mu_{t}(\eta)-t\kappa^-(\eta,i) (=\bar\nu_{t}(\eta)-t\kappa^+(\eta,i)),& \quad \mbox{if $\gamma=\zero,\bgamma=\zero$,}
\end{cases}
\ees
 and
 
 \bes
\pi(\eta',\eta'')=0, \quad \forall (\eta',\eta'')\notin \{(\gamma\eta,\bgamma\gamma_i^+\eta): \gamma,\bgamma\in G^* \}.
\ees 
Then $\bar{\pi}_t\in\Pi(\bar{\mu}_t,\bar\nu_t)$ and $\bar\pi_t$ is optimal for $W_p(\bar\mu_t,\bar\nu_t)$.
\end{lemma}

\begin{proof}
 To check that $\bar\pi_t$ is always non-negative, we remark that \ref{H3.3} implies $\kappa^{\pm}(\eta,i)\geq 0$ and that if $t$ is small enough we have $\bar\mu_{t}(\eta)-t\kappa^-(\eta,i)\geq 0$. It can be verified with a direct calculation that the marginals of $\bar\pi_t$ are $\bar\mu_t$ and $\bar\nu_t$ respectively. To show optimality, we show that the support $\mathrm{supp}(\bar\pi_t)$ of $\bar\pi_t$ is \emph{cyclically monotone}. That is to say, 
\be\label{eq: cyclical monotonicity}
(\eta',\eta''),(\xi',\xi'')\in \mathrm{supp}(\bar\pi_t) \Rightarrow d^p(\eta',\eta'')+d^p(\xi',\xi'') \leq d^p(\xi',\eta'')+d^p(\eta',\xi'').
\ee
Note that by construction, $d(\eta',\eta'')\leq 1$ on $\mathrm{supp}(\bar\pi_t)$. Therefore, if \eqref{eq: cyclical monotonicity} is violated we can w.l.o.g. suppose that $\eta'=\xi''$. This can happen only if $\eta'=\xi''=\eta$ or $\eta'=\xi''=\gamma^+_i\eta$. We show that in the first case \eqref{eq: cyclical monotonicity} is always satisfied. The proof for the second case is almost identical and we omit the details. Therefore, from now on we assume $\eta'=\xi''=\eta$. We observe that 

\begin{align*}
(\eta,\eta'')\in \mathrm{supp}(\bar{\pi})&\Rightarrow \eta'' \in \{\eta,\gamma^{+}_i\eta \}\\
(\xi',\eta)\in \mathrm{supp}(\bar{\pi})&\Rightarrow \xi' \in \{\eta\} \cup \{ \gamma\eta: \gamma\in G, \gamma\neq \gamma^+_{i} \}
\end{align*}
We verify case-by case that \eqref{eq: cyclical monotonicity} holds 

\bei 
\item\underline{$\eta''=\eta,\xi'=\eta$} In this case $d(\eta',\eta'')=d(\xi',\xi'')=d(\xi',\eta'')=d(\eta',\xi'')=0.$

\item \underline{$\eta''=\eta,\xi'=\gamma\eta$ with $\gamma\in G,\gamma\neq\gamma^{+}_i$} In this case 

\bes
d^p(\eta',\eta'')+d^{p}(\xi',\xi'')=d^p(\eta',\xi'')+d^{p}(\xi',\eta'')=d^{p}(\gamma\eta,\eta)=1.
\ees
\item \underline{$\eta''=\gamma^+_i\eta,\xi'=\eta$} In this case 
\bes
d^p(\eta',\eta'')+d^{p}(\xi',\xi'')=d^p(\eta',\xi'')+d^{p}(\xi',\eta'')=d^{p}(\eta,\gamma^+_i\eta)=1.
\ees
\item \underline{$\eta''=\gamma^+_i\eta, \xi'=\gamma\eta$ with $\gamma\in G,\gamma\neq\gamma^+_i$}. In this case $d(\xi',\eta'')=2$ and $d(\eta',\eta''),d(\xi',\xi'')\leq 1$ hold since $(\eta',\eta''),(\xi',\xi'')\in \mathrm{supp}(\bar\pi_t)$.
\eei
Therefore, \eqref{eq: cyclical monotonicity} is never violated and $\bar{\pi}_t$ is optimal for $W_p(\bar\mu_t,\bar\nu_t)$. 

\end{proof}

\section{Glauber dynamics}\label{sec: spin}

In this section, we investigate \eqref{convex sob} for the Glauber dynamics. In view of the applications to classical spin systems such as Curie Weiss or the Ising model, we assume that the moves $\sigma\in G$ are involutions, i.e. $\sigma^{-1}=\sigma$. However, this is not strictly necessary for our method to work. Given an inverse temperature parameter $\beta>0$ and  an Hamiltonian $H:\Omega\longrightarrow\R$ we construct a generator of the form \eqref{def: Markov gen} by setting 

\be\label{eq: spin jump rates}
c(\eta,\sigma) = \exp\Big(-\frac{\beta}{2}\nabla_{\sigma}H(\eta)\Big).
\ee
If the state space $\Omega$ is finite, \ref{H0} and \ref{H1} are satisfied and the reversible measure is the Gibbs measure

\be\label{eq: Gibbs measure}
\inv(\eta)=\frac{1}{Z_{\beta}}\exp(-\beta H(\eta)), \quad \forall \eta\in\Omega,
\ee
where $Z_{\beta}$ is the normalization. Let us now state precisely the assumptions needed for the main result of this section which is Theorem \ref{thm: spin systems} below.

\begin{itemize}
\item[\namedlabel{H3.4}{\textbf{(H3.4)}}] The set of moves $G$ is such that 

\be\label{hyp: involution}
\sigma^{-1}=\sigma, \quad \forall \sigma \in G,
\ee
and the relation
\be\label{hyp: commute}
\sigma\gamma\eta=\gamma\sigma\eta
\ee
holds uniformly on $\eta\in\Omega,\sigma,\gamma\in G$.

\item[\namedlabel{H4.4}{\textbf{(H4.4)}}] $\kappa(\eta,\sigma)\geq 0$ uniformly on $\eta\in\Omega,\sigma\in G$, where

\be\label{kappaspin}
\kappa(\eta,\sigma):= c(\sigma\eta,\sigma) - \sum_{\substack{\gamma\in G \\ \gamma\neq\sigma}} \max\{-\nabla_{\sigma}c(\eta,\gamma),0\}. 
\ee
\end{itemize}
\subsection{Sufficient condition for Glauber dynamics}

\begin{theorem}\label{thm: spin systems}
Let $\beta>0$, $H:\N^d\longrightarrow\R$ be given and the generator $\cL$ be defined by \eqref{eq: spin jump rates}. Moreover, assume that \ref{H3.4} and \ref{H4.4} hold. If we define 
\be\label{ass: kappa spin}
\kappa= \inf_{\substack{\eta\in\Omega \\ \sigma\in G}}   \kappa(\eta,\sigma)+\kappa(\sigma\eta,\sigma), \quad \bar{\kappa} = \inf_{\substack{\eta\in\Omega \\ \sigma\in G}}   \kappa(\eta,\sigma)
\ee

then the following holds 

 \item[(i)] For any $\phi$ satisfying \ref{H2} the convex Sobolev inequality \eqref{eq: convex sob rev} holds with with $\kappa_{\phi}=\kappa$. 
  \item[(ii)] The modified Logarithmic Sobolev inequality \eqref{log sob} holds with $\kappa_1=\kappa+2\bar{\kappa}$. 
   \item[(iii)] For $\alpha \in (1,2]$, the Beckner inequality \eqref{beckner ineq} holds with $\kappa_{\alpha}=\alpha\kappa$.
\end{theorem}

The assumptions and the proof of Theorems \ref{thm: LSI int random walks} and \ref{thm: spin systems} bear several resemblances. We could indeed merge them in a single general result. However, for the sake of clarity we prefer to keep the two results distinct.

\begin{proof} The proof is done verifying that the assumptions of Proposition \ref{cor: sufficient cond} hold with $\kappa'=\kappa'''=\kappa$ and $\kappa''=\bar\kappa$ for the coupling rates we are going to construct. The proof is then finished applying the same proposition. For any $\eta\in\Omega$ we define $\Upsilon^{<}(\eta),\Upsilon^{>}(\eta),\Upsilon^{=}(\eta)$ as follows
\begin{align*}
&\Upsilon^{<}(\eta)= \{ (\sigma,\gamma)\in G\times G: \sigma\neq\gamma, \nabla_{\sigma}c(\eta,\gamma)<0 \}\\
&\Upsilon^{>}(\eta)= \{ (\sigma,\gamma)\in G\times G: \sigma\neq\gamma,\nabla_{\sigma}c(\eta,\gamma)>0 \}\\
&\Upsilon^{=}(\eta)= \{ (\sigma,\gamma)\in G\times G: \sigma\neq\gamma, \nabla_{\sigma}c(\eta,\gamma)=0 \}.
\end{align*}

We remark that under the current hypothesis we have $\cS=\Omega\times G$. For any $\eta\in\Omega,\sigma\in G$ we define

\be\label{eq: spin coup rates}
\couprate(\eta,\sigma\eta,\gamma,\bgamma)=\begin{cases}
\min\{c(\sigma \eta,\gamma),c(\eta,\gamma)\}, & \quad \mbox{ if $\gamma=\bgamma$ and $\sigma\neq\gamma, \gamma\in G$,}\\
-\nabla_{\sigma}c(\eta,\gamma), & \quad \mbox{if $\bgamma=\sigma$ and $(\sigma,\gamma)\in \Upsilon^<(\eta)$,}\\
\nabla_{\sigma}c(\eta,\bgamma), & \quad \mbox{if $\gamma=\sigma$ and $(\sigma,\bgamma)\in \Upsilon^>(\eta)$,}\\
\kappa(\sigma\eta,\sigma), & \quad \mbox{if $\gamma=\sigma,\bgamma=\zero$,}\\
\kappa(\eta,\sigma), & \quad \mbox{if $\gamma=\zero,\bgamma=\sigma$,}\\
0,& \quad \mbox{otherwise.}\\
\end{cases}
\ee
A direct calculation  using \ref{H3.4} and \ref{H4.4} shows that \eqref{eq: spin coup rates} define admissible coupling rates. In particular, \ref{H4.4} ensures that $\couprate(\eta,\sigma\eta,\sigma,\zero)$ and $\couprate(\eta,\sigma\eta,\zero,\sigma)$ are non negative. We now show that \eqref{eq: discrete gradient estimate} holds with $\kappa'=\kappa$. To this aim observe that for any  $f>0$ and $\phi$ satisfying \ref{H2} the choice \eqref{eq: spin coup rates} give that the left hand side of \eqref{eq: discrete gradient estimate} rewrites as
$\frac{1}{2}(A+B+C+D)$ with

\begin{align*}
A &= \sum_{\substack{\eta\in\Omega, \sigma,\gamma\in G \\ \sigma\neq\gamma}} c(\eta,\sigma)\min\{c(\eta,\gamma),c(\sigma\eta,\gamma)\} \coupgrad_{\gamma,\gamma} \fphi(\eta,\sigma\eta)\inv(\eta), \\
B &= -\sum_{\substack{\eta\in\Omega\\ (\sigma,\gamma)\in\Upsilon^<(\eta)}} c(\eta,\sigma)\nabla_{\sigma}c(\eta,\gamma) \coupgrad_{\gamma,\sigma} \fphi(\eta,\sigma\eta)\inv(\eta), \\
C &= \sum_{\substack{\eta\in\Omega \\(\sigma,\bgamma)\in\Upsilon^>(\eta)}} c(\eta,\sigma)\nabla_{\sigma}c(\eta,\bgamma) \coupgrad_{\sigma,\bgamma} \fphi(\eta,\sigma\eta)\inv(\eta), \\
D &= \sum_{\eta\in\Omega,\sigma\in G} c(\eta,\sigma)[\kappa(\sigma\eta,\sigma)\coupgrad_{\sigma,\zero}\fphi(\eta,\sigma\eta)+\kappa(\eta,\sigma) \coupgrad_{\zero,\sigma}\fphi(\eta,\sigma\eta)]\inv(\eta).
\end{align*}
We now show that $A=B=C=0$. We begin by considering $B$. Using \ref{H3.4} and \eqref{eq: spin jump rates} we get that for all $\sigma\neq\gamma$
\bes
(\sigma,\gamma)\in\Upsilon^<(\eta) \Leftrightarrow \nabla_{\sigma}\nabla_{\gamma} H(\eta) > 0\Leftrightarrow (\gamma,\sigma)\in\Upsilon^<(\eta).
\ees
Therefore we can rewrite $B$ exchanging the labels $\sigma$ and $\gamma$ as
\bes
  - \frac{1}{2} \sum_{\substack{\eta\in\Omega\\ (\sigma,\gamma)\in\Upsilon^<(\eta)}} [c(\eta,\sigma)\nabla_{\sigma}c(\eta,\gamma)-c(\eta,\gamma)\nabla_{\gamma}c(\eta,\sigma)] [\fphi(\eta,\gamma\eta)-\fphi(\eta,\sigma\eta)]\inv(\eta)
\ees
from which $B=0$ follows. Indeed, \eqref{eq: spin jump rates} implies that $c(\eta,\sigma)\nabla_{\sigma}c(\eta,\gamma)=c(\eta,\gamma)\nabla_{\gamma}c(\eta,\sigma)$ for all $\eta,\gamma,\sigma$.
Using \eqref{eq rev} on $C$ with 
\[ F(\eta,\sigma)=\sum_{\bgamma:(\sigma,\bgamma)\in\Upsilon^>(\eta)} \nabla_{\sigma}c(\eta,\bgamma) \coupgrad_{\sigma,\bgamma} \fphi(\eta,\sigma\eta) \]
and \ref{H3.4} yields the equivalent expression
\bes
 -\sum_{\substack{\eta\in\Omega\\(\sigma,\bgamma)\in \Upsilon^{>}(\eta)}} c(\eta,\sigma)\nabla_{\sigma}c(\eta,\bgamma) [\fphi(\eta,\bgamma\eta)-\fphi(\eta,\sigma\eta)]\inv(\eta).
\ees
Using the same argument used to show $B=0$, we conclude that $C=0$. The proof that $A=0$ is done in the auxiliary Lemma \ref{lem cancel spin} and follows from reversibility. Finally, recalling that $\fphi$ vanishes on the diagonal we easily get that 
$D\leq -2\kappa\cE(\phi'(f),f)$, which completes the proof that \eqref{eq: discrete gradient estimate} holds with $\kappa'=\kappa.$  From the construction of the coupling rates we also have that \eqref{def: kappa''} holds with $\kappa''=\bar\kappa$ and \eqref{def: kappa'''} holds with $\kappa'''=\kappa$. An application of Proposition \ref{cor: sufficient cond} finishes the proof.
\end{proof}

\begin{lemma}\label{lem cancel spin}

Under the hypothesis of Theorem \ref{thm: spin systems} we have
\be\label{eq: cancel spin}
\sum_{\substack{\eta\in\Omega, \sigma,\gamma\in G \\ \sigma\neq\gamma}} c(\eta,\sigma)\min\{c(\eta,\gamma),c(\sigma\eta,\gamma)\} \coupgrad_{\gamma,\gamma} \fphi(\eta,\sigma\eta)\inv(\eta)=0
\ee
holds for all $f>0$.
\end{lemma}

\begin{proof}
Recalling the definition of $\Upsilon^<(\eta),\Upsilon^=(\eta),\Upsilon^>(\eta)$, we rewrite \eqref{eq: cancel spin} as the sum of the three terms 

\begin{align}\label{eq: cancel spin split up}
&\nonumber \sum_{\substack{\eta\in\Omega \\ (\sigma,\gamma)\in\Upsilon^{<}(\eta)}} c(\eta,\sigma)c(\sigma\eta,\gamma) \coupgrad_{\gamma,\gamma} \fphi(\eta,\sigma\eta)\inv(\eta) \\
&\nonumber \sum_{\substack{\eta\in\Omega \\ (\sigma,\gamma)\in\Upsilon^{>}(\eta)}}  c(\eta,\sigma)c(\eta,\gamma) \coupgrad_{\gamma,\gamma} \fphi(\eta,\sigma\eta)\inv(\eta)\\
&\frac{1}{2}\sum_{\substack{\eta\in\Omega \\ (\sigma,\gamma)\in\Upsilon^{=}(\eta)}} c(\eta,\sigma)[c(\eta,\gamma)+c(\sigma\eta,\gamma)] \coupgrad_{\gamma,\gamma} \fphi(\eta,\sigma\eta)\inv(\eta).
\end{align}

Using the reversibility \eqref{eq rev} and \ref{H3.4} on the second term with with

\bes
F(\eta,\gamma)=\sum_{\sigma: (\sigma,\gamma)\in\Upsilon^{>}(\eta)} c(\eta,\sigma)\coupgrad_{\gamma,\gamma} \fphi(\eta,\sigma\eta)
\ees

yields the equivalent form

\begin{align}\label{eq: spin cancel rev}
-\sum_{\substack{\eta\in\Omega \\ (\sigma,\gamma)\in\Upsilon^{>}(\gamma\eta)}}  c(\eta,\gamma)c(\gamma\eta,\sigma) \coupgrad_{\gamma,\gamma} \fphi(\eta,\sigma\eta)\inv(\eta)
\end{align}
Next observe that, thanks to \eqref{eq: spin jump rates} and \ref{H3.4} we obtain
\bes
(\sigma,\gamma)\in\Upsilon^>(\gamma\eta)\Leftrightarrow\nabla_\gamma\nabla_\sigma H(\gamma\eta)>0\Leftrightarrow\nabla_\gamma\nabla_\sigma H(\eta)<0\Leftrightarrow(\sigma,\gamma) \in \Upsilon^<(\eta).
\ees
Plugging this back into \eqref{eq: spin cancel rev}, we obtain that the first two terms in \eqref{eq: cancel spin split up} cancel. Arguing as for the first two terms and observing that $\Upsilon^=(\eta)=\Upsilon^=(\gamma\eta)$ we obtain tha the third term is also worth $0$.
\end{proof}

\subsection{Applications to spin systems}\label{sec: spin examples}

\subsubsection{Curie-Weiss model}\label{sec: curie}

For the Curie-Weiss model we have $\Omega=\{1,1\}^N$ for some $N>0$ and the set of moves is $G=\{\sigma_i\}_{i=1,\ldots,N}$. $\sigma_i$ acts on $\eta$ flipping its $i-th$ coordinate, i.e.

\be\label{def: flip moves}
\sigma_i(\eta)_j=\begin{cases} \eta_j, & \quad \mbox{if $j\neq i$} \\
-\eta_j, & \quad \mbox{if $j= i$.}
\end{cases}
\ee
The Hamiltonian is given by
\bes
H:\{-1,1\}^N\rightarrow\R,\quad  H(\eta)= -\frac{1}{2N} \sum_{i,j=1}^N \eta_i\eta_j.
\ees
For a given $\beta>0$, the transition rates of the Glauber dynamics are then given by 

\be\label{eq: curie rates}
c(\eta,\sigma_i)= \exp\big(-\beta \eta_{i}\, m_i(\eta)\big), \quad \text{with} \,\, m_i(\eta)= \frac{1}{N} \sum_{j\neq i} \eta_j.
\ee

To state our result for the Curie-Weiss model, it is convenient to introduce $f_{CW,\beta,N}:\N\rightarrow\R$

\beas
f_{\mathrm{CW},\beta,N}(m) :=  \exp\Big( -\frac{\beta}{N}(N-1-2m) \Big)[1-(N-1-m)(\exp\Big(\frac{2\beta}{N}\Big)-1)] \\
+  \exp\Big( \frac{\beta}{N}(N-1-2m) \Big)[1-m(\exp\Big(\frac{2\beta}{N}\Big)-1)]
\eeas

\begin{theorem}
Assume that
\be\label{hyp: curie condition} (N-1)(\exp(2\beta/N)-1)\leq 1.
\ee

Then the conclusion of Theorem \ref{thm: spin systems} holds with 

\begin{align}\label{eq: curie kappa}\nonumber\kappa&=  f_{\mathrm{CW},\beta,N}(\lfloor(N-1)/2 \rfloor),\\
\bar{\kappa}&=\exp(-\frac{\beta}{N}(N-1))[1-(N-1)(1-\exp(2\beta/N))].
\end{align}
In particular, if $N$ is odd, we have $f_{CW,\beta,N}(\lfloor (N-1)/2 \rfloor) = 2\left(1-\frac{N-1}{2}(\exp(2\beta/N)-1)\right)$.
\end{theorem}

\begin{remark}
As $N\rightarrow+\infty$, our condition \eqref{hyp: curie condition} reads as $\beta\leq 1/2$, thus improving on \cite[Cor 4.5]{erbar2017ricci}. Estimates on the best constant for the classical (i.e. non modified) logarithmic Sobolev inequality that are valid up to the critical temperature $\beta=1$ have been obtained in \cite{marton2015logarithmic}. For $N$ large the MLSI constant we obtain from Theorem \ref{thm: spin systems} is $2(1-\beta)+2(1-2\beta)\exp(-\beta)$ improving on the value $4(1-2\beta\exp(2\beta))\exp(-\beta)$ found in \cite{erbar2017ricci}. The findings of Theorem 4.1 concerning general convex Sobolev inequalities and Beckner's inequalities seem to be new.
\end{remark}

\begin{proof}
We obtain from \eqref{eq: curie rates} that for all $\eta,i,j\neq i$:
\bes
\frac{c(\sigma_i\eta,\sigma_j)}{c(\eta,\sigma_j)}= \exp\big(\frac{2\beta \eta_i\eta_j}{N}\big)
\ees

Moreover, if $|\{ j\neq i: \eta_i\eta_j=1\}|=m$ we have

\bes
c(\sigma_i\eta,\sigma_i) = \exp\big( -\frac{\beta}{N}(N-1-2m) \big)
\ees
Therefore
\bes
\kappa(\eta,\sigma_i) = \exp\Big( -\frac{\beta}{N}(N-1-2m) \Big)- \sum_{\substack{j:\eta_j \eta_i=-1\\j\neq i}} (\exp\big(\frac{2\beta}{N}\big)-1)c(\sigma_i\eta,\sigma_j)
\ees
Next, we observe that if $\eta_i\eta_j=-1$, we have that $c(\sigma_i\eta,\sigma_j)=c(\sigma_i\eta,\sigma_i)$. Since there are $N-1-m$ spins of this type, we obtain
\be\label{eq: kappa bound curie}
\kappa(\eta,\sigma_i) =\exp\Big( -\frac{\beta}{N}(N-1-2m) \Big) \Big[ 1-(N-1-m)(\exp\big(\frac{2\beta}{N}\big)-1) \Big]
\ee
In particular, we obtain that \eqref{hyp: curie condition} implies \eqref{kappaspin} and we can apply Theorem \ref{thm: spin systems}. From \eqref{eq: kappa bound curie} we also obtain that $\bar\kappa$ therein is given by \eqref{eq: curie kappa}. To finish the proof, observe that 
\begin{align*}
\kappa(\eta,\sigma_i) + \kappa(\sigma_i\eta,\sigma_i) &= \exp\Big( -\frac{\beta}{N}(N-1-2m) \Big)[1-(N-1-m)(\exp\Big(\frac{2\beta}{N}\Big)-1)]\\
&+  \exp\Big( \frac{\beta}{N}(N-1-2m) \Big)[1-m(\exp\Big(\frac{2\beta}{N}\Big)-1)] =f_{\mathrm{CW},\beta,N}(m).
\end{align*}

The right hand side being a convex function $m\in[0,N-1]$ and symmetric around $m=(N-1)/2$, it achieves its minimum on $\{0,\ldots,N-1\}$ at $m=\lfloor(N-1)/2\rfloor$.
\end{proof}

\subsubsection{Ising model}\label{sec: Ising}

Let $\Lambda\subseteq\Z^d$ be a connected subset, which we endow with the natural graph structure $\sim$ inherited from $\Z^d$. The state space is $\Omega=\{-1,1\}^{\Lambda}$ and  the Hamiltonian is

\bes
H:\{-1,1\}^{\Lambda} \longrightarrow\R,\quad H(\eta)=\frac12\sum_{x\sim y}\eta_x\eta_y.
\ees
where $x\sim y$ means that $x$ and $y$ are neighbors in $\Lambda$. The set of moves is $G=\{ \sigma_x\}_{x\in \Lambda}$, where $\sigma_x$ is the flip of the spin at site $x$, see \eqref{def: flip moves}. Therefore, for $\beta>0$ the transition rates \eqref{eq: spin jump rates} for the Glauber dynamics are

\bes
c(\eta,\sigma_x)= \exp\left(-\beta\eta_x \sum_{y\sim x} \eta_y\right)
\ees 

 \begin{theorem}\label{thm: Ising}
Assume that 

 \be\label{hyp: Ising}
 2d (1-\exp(-2\beta))\exp(4d\beta)\leq 1 
 \ee
 
 Then the conclusion of Theorem \ref{thm: spin systems} holds with 
 
 \be\label{eq: kappa Ising} \kappa= 2-2d (1-\exp(-2\beta))\exp(2\beta d)\ee
 and
\be\label{eq: barkappa Ising} \bar{\kappa}=\exp(-2\beta d)-2d(1-\exp(-2\beta))\exp(2\beta d). \ee
\end{theorem}

\begin{remark}
In the article \cite[Cor 4.4]{erbar2017ricci} the authors establish entropic Ricci curvature bounds, and in particular MLSI, under the condition $\varepsilon(\beta)\leq1$, where 

\bes
\varepsilon(\beta) =(2d-1) (1-\exp(-2\beta))\exp(4\beta d)
\ees
The condition \eqref{hyp: Ising} of Theorem \ref{thm: Ising} is therefore more demanding. However, when both results apply, the bound on the MLSI constant $\kappa+2\bar\kappa$ provided by Theorem \ref{thm: Ising} is better than the bound $4(1-\varepsilon(\beta))\exp(-2\beta d)$ found there, at least for $d\geq 2$. Indeed, after some calculations, one can find that the difference between the two bounds is

\bes 2(1-\exp(-2\beta d)) + (4-3\frac{2d}{2d-1})\varepsilon(\beta)\exp(-2\beta d).\ees
As for the Curie Weiss, the findings of Theorem \ref{thm: Ising} concerning Beckner inequalities and general convex Sobolev inequalities seem to be new and estimates on the best constant for the non modified logarithmic Sobolev inequality have been obtained in \cite{marton2015logarithmic} under a condition that appears to be weaker than \eqref{hyp: Ising}.
\end{remark}

\begin{proof}
The proof is done verifying that the hypothesis of Theorem \ref{thm: spin systems} hold with $\kappa,\bar\kappa$ as in \ref{eq: kappa Ising} and \eqref{eq: barkappa Ising}. We first observe that for all $x\neq y$ we have
\bes
\frac{c(\sigma_x\eta,\sigma_y)}{c(\eta,\sigma_y)} = \begin{cases}
\exp(2\beta\eta_x\eta_y) & \quad \mbox{if $x\sim y$}\\
 1, & \quad \mbox{otherwise} 
\end{cases},
\ees
 and that if $|\{y\sim x :\eta_x\eta_y=1 \}|=m\in\{0,\ldots,2d\}$ then we have
 
 \bes
c(\eta,\sigma_x) = \exp(2\beta(d-m)), \quad  c(\sigma_x\eta,\sigma_x) = \exp(2\beta(m-d)).
\ees

Therefore, Recalling the deinition \eqref{kappaspin} of $\kappa(\eta,\sigma_x)$  we have 
\begin{align*}
\kappa(\eta,\sigma_x)&=c(\sigma_x\eta,\sigma_x)+\sum_{\substack{y\sim x \\ \eta_x\eta_y=-1 }} c(\sigma_x\eta,\sigma_y)-c(\eta,\sigma_y) \\
&=c(\sigma_x\eta,\sigma_x)+(1-\exp(2\beta))\sum_{\substack{y\sim x \\ \eta_x\eta_y=-1 }} c(\sigma_x\eta,\sigma_y)
\end{align*}
If $y\sim x$ with $\eta_x\eta_y=-1$ then
\begin{align*}
c(\sigma_x\eta,\sigma_y) &= \exp(-\beta \eta_y \sum_{z\sim y}(\sigma_x\eta)_z) \\
&= \exp\left(-\beta \eta_y \sum_{\substack{z\sim y \\ z\neq x }}\eta_z +\beta\eta_y\eta_x\right)\leq \exp(2\beta(d-1)).
\end{align*}

and therefore 

\be\label{eq: estimate on kappa ising}
\kappa(\eta,\sigma_x)\geq \exp(2\beta(m-d))-(1-\exp(-2\beta))(2d-m)\exp(2\beta d).
\ee

In particular, if \eqref{hyp: Ising} holds, then \ref{H4.4} is satisfied and we can apply  Theorem \ref{thm: spin systems}. It remains to compute $\kappa$ and $\bar{\kappa}$. From \eqref{eq: estimate on kappa ising} we immediately get that $\bar\kappa$ can be taken as in \eqref{eq: barkappa Ising}. Using the same argument that led to \eqref{eq: estimate on kappa ising} one gets

\bes
\kappa(\sigma_x\eta,\sigma_x)\geq \exp(2\beta(d-m))-(1-\exp(-2\beta))m\exp(2\beta d)
\ees

If $\kappa$ is given by \eqref{eq: kappa Ising}, observing that $a+1/a\geq 2$ we get that $\kappa(\eta,\sigma_x)+\kappa(\sigma_x\eta,\sigma_x) \geq \kappa$ holds uniformly in $\eta\in\Omega,x\in\Lambda$. The conclusion follows by Theorem \ref{thm: spin systems}.
\end{proof}

\section{More examples}\label{sec: review}

\subsection{Bernoulli Laplace}
The Bernoulli Laplace model is the simple exclusion process on the complete graph. Given $L>N\in\N$, where $L$ represents the number of sites and $N$ the number of particles we consider the state space
\bes
\Omega= \Big\{ \eta:\{1,\ldots,L \} \rightarrow \{0,1\}: \sum_{i=1}^L\eta_i=N \Big\},
\ees 
where $\eta_i=1$ means that a particle is present at site $i$. For any $1\leq i \leq L$ we define $\delta_i\in\Omega$ by 
\bes 
(\delta_i)_k = \begin{cases}
1, & \quad \mbox{if $i=k$,}\\
0, & \quad \mbox{otherwise.}
\end{cases}
\ees
The set of moves is $G=\{\sigma_{ij}, i,j\in\{1,\ldots,L\} \}$, where 
\bes
\sigma_{ij}(\eta)=\begin{cases} \eta - \delta_i +\delta_j, & \quad \mbox{if $\eta_i(1-\eta_j)=0$,}\\
\eta, & \quad \mbox{otherwise}
\end{cases}
\ees
The map $\sigma_{ij}$ moves a particle from site $i$ to site $j$, when this is possible. The jump rates for the Bernoulli-Laplace model are defined by
\bes
\forall i,j\in\{1,\ldots,L\}, \quad c(\eta,\sigma_{ij}) =\eta_i(1-\eta_j),
\ees
and therefore the reversible measure $\inv$ is the uniform measure on $\Omega$. Assumptions \ref{H0} and \ref{H1} are clearly satisfied. In particular, $(\sigma_{ij})^{-1}=\sigma_{ji}$.

\begin{theorem}\label{thm: BL}
For the Bernoulli Laplace model the the following hold
\bei
\item[(i)] If $\phi$ satisfies \ref{H2}, then the convex Sobolev inequality \eqref{eq: convex sob rev} holds with $\kappa_{\phi}=L$. 
  \item[(ii)] The modified log Sobolev inequality \eqref{log sob} holds with $\kappa_1=L+2$.
 \item[(iii)] For $\alpha \in (1,2]$, the Beckner inequality \eqref{beckner ineq} holds with $\kappa_{\alpha}=\alpha L$.
 \eei
 \end{theorem}
\begin{remark}
Poincaré inequalities and MLSI for the Bernoulli Laplace model have been extensively studied, see \cite{gao2003exponential,goel2004modified,bobkov2006modified}. The estimate ons $\kappa_1$ and $\kappa_2$ given by Theorem \ref{thm: BL} match the best known results \cite{caputo2009convex,erbar2015discrete,diaconis1987time}. Beckner inequalites have been studied in \cite{bobkov2006modified} and \cite{jungel2017discrete}. Our constant agrees with the one found in \cite{jungel2017discrete}. In there, the more general case of non-homogeneous rates is treated as well. Arguably, our method also works in this case but we leave it to future work to verify this. For general functions $\phi$ satisfying \ref{H2} the convex Sobolev inequality obtained at Theorem \ref{thm: BL} seems to be new.
\end{remark}

\begin{proof}
Let $(\eta,\sigma_{ij})\in\cS$, i.e. $\eta_i=1,\eta_j=0.$ We define

\be\label{eq: coup rates BL}
\couprate(\eta,\sigma_{ij}\eta,\gamma,\bgamma)=\begin{cases}  \min\{c(\eta,\gamma),c(\sigma_{ij}\eta,\gamma) \}, & \quad \mbox{if $\gamma=\bgamma \in G$}, \\
1, & \quad \mbox{if $\gamma=\sigma_{ij},\bgamma=\zero$ or $\gamma=\zero,\bgamma=\sigma_{ji}$},\\
(1-\eta_l), & \quad \mbox{if $\gamma=\sigma_{il},\bgamma=\sigma_{jl}$, $l\notin\{i,j\}$},\\
\eta_k,  & \quad \mbox{if $\gamma=\sigma_{kj},\bgamma=\sigma_{ki}$, $k\notin\{i,j\}$}.\\
0,& \quad \mbox{otherwise}\\
\end{cases}
\ee
It can be verified with a direct calculation that \eqref{eq: coup rates BL} defines admissible coupling rates. Let  $\phi$ satisfy \ref{H2} and $f>0$. In view of \eqref{eq: coup rates BL}, the left hand side of \eqref{eq: discrete gradient estimate} can be written as $\frac{1}{2}(A+B+C+D)$ with
\begin{align*}
A&=\sum_{\substack{\eta,(i,j),(k,l)}}c(\eta,\sigma_{ij})\min\{c(\eta,\sigma_{kl}),c(\sigma_{ij}\eta,\sigma_{kl})\}\coupgrad_{\sigma_{kl},\sigma_{kl}} \fphi(\eta,\sigma_{ij}\eta)\inv(\eta),\\
B&=\sum_{\eta,(i,j)} c(\eta,\sigma_{ij})[\coupgrad_{\sigma_{ij},\zero} \fphi(\eta,\sigma_{ij}\eta) + \coupgrad_{\zero,\sigma_{ji}} \fphi(\eta,\sigma_{ij}\eta) ]\inv(\eta),\\
C&=\sum_{\eta, (i,j)} c(\eta,\sigma_{ij})[\sum_{l\neq i,j}(1-\eta_l) \coupgrad_{\sigma_{il},\sigma_{jl}} \fphi(\eta,\sigma_{ij}\eta)]\inv(\eta),\\
D&=\sum_{\eta, (i,j)} c(\eta,\sigma_{ij})[\sum_{k\neq i,j}\eta_k \coupgrad_{\sigma_{kj},\sigma_{ki}} \fphi(\eta,\sigma_{ij}\eta)]\inv(\eta).\\
\end{align*}
We show that $A=0$. To do this, we first observe that 

\bes
c(\eta,\sigma_{ij})\min\{c(\eta,\sigma_{kl}),c(\sigma_{ij}\eta,\sigma_{kl})\} = 
\begin{cases} 
1, & \quad \mbox{if $i\neq k$,$\eta_i=\eta_k=1$,$\eta_j=\eta_l=0,j\neq l$,}\\
0, & \quad \mbox{otherwise.}
\end{cases}
\ees
Therefore, $c(\eta,\sigma_{ij})\min\{c(\eta,\sigma_{kl}),c(\sigma_{ij}\eta,\sigma_{kl})\}=c(\eta,\sigma_{kl})\min\{c(\eta,\sigma_{ij}),c(\sigma_{kl}\eta,\sigma_{ij})\}$ and we can rewrite $A$ as 
\bes
\sum_{\substack{\eta,(k,l)}}c(\eta,\sigma_{kl})F(\eta,\sigma_{kl})\inv(\eta)
\ees
with
\bes
F(\eta,\sigma_{kl})=\sum_{(i,j)}\min\{c(\eta,\sigma_{ij}),c(\sigma_{kl}\eta,\sigma_{ij})\}\coupgrad_{\sigma_{kl},\sigma_{kl}} \fphi(\eta,\gamma_{ij}\eta).
\ees
Using \eqref{eq rev} on this last expression we then obtain $A=-A$, whence $A=0$. Using that $\fphi$ vanishes on the diagonal and the fact that any $\eta\in\Omega$ has $N$ particles occupying $L$ sites and that $(\eta,\sigma_{ij})\in \cS$ implies $\eta_i=1,\eta_j=0$:

\begin{align*}
    B &=-4 \cE(\phi'(f),f),\\
    C &= -2(L-N-1)\cE(\phi'(f),f),\\
    D&=-2(N-1)\cE(\phi'(f),f).
\end{align*}
Therefore \eqref{eq: discrete gradient estimate} holds with $\kappa'=L$. Moreover, from the construction of coupling rates we get that \eqref{def: kappa''} holds with $\kappa''=1$ and \eqref{def: kappa'''} holds with $\kappa'''=L$. The conclusion follows from Proposition \ref{cor: sufficient cond}.
\end{proof}

\subsection{The hardcore model}\label{sec: hardcore}

Consider a finite undirected graph $(V,E)$ that is also simple ($(x,x)\notin E$) and connected. As usual, if $(x,y)\in E$ we write $x\sim y$ and say that $x,y$ are neighbors. The state space of the classical hardcore model is 

\bes \Omega=\{\eta: V\rightarrow\{0,1\} \, \text{s.t.} \, \eta_x\eta_y=0, \, \forall  x\sim y \}. \ees
For $x\in V$ we define its neighborhood as $N_x=\{y\neq x: y\sim x \}$ and we set $\bar{N}_x=N_x\cup\{x\}$. A configuration $\eta\in\Omega$ is such that if a site $x$ is occupied, then all sites in its neighborhood are empty. For any $x\in V$ we define $\delta_x\in\Omega$ as 
\bes
(\delta_x)_y= \begin{cases} 1, & \quad \mbox{if $y=x$ } \\
0, & \quad \mbox{otherwise.}
\end{cases}
\ees
The set of moves is $G=\{ \gamma_x^+,\gamma^-_x \, : \, x\in V\}$, where 

\bes
\gamma_x^+(\eta)=\begin{cases}\eta+\delta_x ,\quad \mbox{if $\eta+\delta_x\in\Omega$} \\ \eta \quad \mbox{otherwise} \end{cases}, \gamma_x^-(\eta)=\begin{cases}\eta-\delta_x ,\quad \mbox{if $\eta+\delta_x\in\Omega$} \\ \eta \quad \mbox{otherwise} \end{cases}.
\ees
The generator is given by

\be\label{eq: hardcore gen}
\cL f (\eta) = \sum_{x\in V}c(\eta,\gamma_x^-)\nabla_{x}^-f(\eta) + c(\eta,\gamma^+_x)\nabla^{+}_{x}f(\eta)
\ee
where 
\bes
c(\eta,\gamma_x^-)=\eta_x,\quad c(\eta,\gamma_x^+) = \rho\prod_{y\in \bar{N}_x}(1-\eta_y).
\ees
for some constant $\rho>0$. The above means that a new particle arrives at rate $\rho$ on an empty site $x\in V$ the neighborhood of $x$ is empty. Each occupied site $x\in V$ is emptied at rate $1$. It is clear that \ref{H0} and \ref{H1} hold. In particular, for all $x\in V$ we have $(\gamma^+_x)^{-1}=\gamma_x^-$ and the reversible measure for the hardcore model is known to be (see \cite{dai2013entropy} for example)

\bes
\pi(\eta)= \frac{1}{Z}\mathbf{1}_{\eta\in\Omega} \prod_{x\in V}\rho^{\eta_x},
\ees
where $Z$ is the normalization.
 
\begin{theorem}\label{thm: hardcore}
Let $\Delta$ be the maximum degree of $(V,E)$,  $\Delta=\sup_{x\in V} |N_x|$. Assume that

\be\label{hardcore ass}
\rho\Delta \leq 1
\ee

and set

\be\label{eq:kappa hard} \kappa=1-\rho(\Delta-1), \quad \bar{\kappa}= \min\{\rho,(1-\rho\Delta)\} \ee
Then the following hold
\bei
\item[(i)] If $\phi$ satisfies \ref{H2}, then the convex Sobolev inequality \eqref{eq: convex sob rev} holds with $\kappa_{\phi}=\kappa$. 
 \item[(ii)] The modified log Sobolev inequality \eqref{log sob} holds with $\kappa_1=\kappa+2\bar{\kappa}$.
 
\item[(iii)] For $\alpha \in (1,2]$, the Beckner inequality \eqref{beckner ineq} holds with $\kappa_{\alpha}=\alpha\kappa$
 
 \eei

\end{theorem}
The hardcore model and its generalizations have been intensively studied, see the discussion in \cite[Sec. 22.4]{levin2017markov} of the book by Levin, Peres and Wilmer. Mixing times have been studied in \cite{luby1999fast},\cite{dyer2000markov},\cite{vigoda2001note} among others.

\begin{remark}
The best estimates for the MLSI of the hardcore model have been obtained in \cite{dai2013entropy,erbar2017ricci}. For instance, in \cite[Cor 4.8]{erbar2017ricci} MLSI is shown to old with constant $1-\rho(\Delta-1)$ under assumption \eqref{hardcore ass}. Therefore \eqref{thm: hardcore} improves on this result. We are not aware of previously known results about Beckner inequalities or general convex Sobolev inequalities for the hardcore model. Note that in the two above mentioned references a more general version of the hardcore model is considered. We leave it to future work to see whether the methods of this paper yield interesting results for the general model.
\end{remark}

\begin{proof}
The proof is done constructing coupling rates such that the assumptions of Proposition \ref{cor: sufficient cond} hold with $\kappa'=\kappa'''=\kappa$, $\kappa''=\bar\kappa$. Consider a pair $(\eta,\gamma_x^+)\in \cS$. This means that $x$ and all sites in the neighborhood of $x$ are empty, i.e. $\eta|_{\bar{N}_x} \equiv 0$. We then define
\be\label{eq: hard coup rates 1}
    \couprate(\eta,\gamma_x^+\eta,\gamma,\bgamma) = \begin{cases}  \min\{c(\eta,\gamma),c(\gamma_x^+\eta,\gamma) \} , &\quad \mbox{if $\gamma=\bgamma\in G$,}\\
    \rho, & \quad \mbox{if $\gamma=\gamma_y^+,\bgamma=\gamma_x^-$ with  $y\sim x,\eta|_{\bar{N}_y} \equiv 0$,}\\
    \rho,  & \quad \mbox{if $\gamma=\gamma_x^+$,$\bgamma=\zero$,}\\
    1-\rho \big|\{y: y\sim x, \eta|_{\bar{N}_y}\equiv 0 \}\big|,  & \quad \mbox{if $\gamma=\zero$,$\bgamma=\gamma^-_x$,}\\
    0, & \quad \mbox{otherwise.}
    \end{cases}
\ee
If $(\eta,\gamma^{-}_x)\in \cS$, then $(\gamma^-_x\eta,\gamma^{+}_x)\in \cS$ as well and using the former definition we set
\be\label{eq: hard coup rates 2} 
\forall\gamma,\bgamma\in G^*, \quad \couprate(\eta,\gamma^{-}_x\eta,\gamma,\bgamma)=\couprate(\gamma_x^-\eta,\gamma_x^+(\gamma^-_x\eta),\bgamma,\gamma)=\couprate(\gamma_x^-\eta,\eta,\bgamma,\gamma) . \ee
Thanks to \eqref{hardcore ass} we have that $\couprate(\eta,\gamma_x^+\eta,\gamma,\bgamma)$ is always non negative. It can verified with a direct calculation that \eqref{eq: hard coup rates 1},\eqref{eq: hard coup rates 2} define admissible coupling rates. In order to do so, observe that $c(\eta,\gamma^+_{y})=c(\gamma^+_{x}\eta,\gamma^+_{y})$ as soon as $y\notin\bar{N}_x$ and that $c(\eta,\gamma^-_{y})=c(\gamma_x^+\eta,\gamma^-_{y})$ for all $y\in V$ except for $y=x$ where $c(\eta,\gamma^-_{x})=0$ and $c(\gamma_x^+\eta,\gamma^-_{x})=1$. The next step is to prove that \eqref{eq: discrete gradient estimate} holds with $\kappa'=\kappa$. We have to show that for all $f>0$ and $\phi$ satisfying \ref{H2}

\begin{align}\label{eq: hardcore G+G-}
\nonumber \frac12\sum_{(\eta,\gamma_x^+)\in\supp}\sum_{(\gamma,\bgamma)\in G^*}c(\eta,\gamma_x^+)\couprate(\eta,\gamma_x^+\eta,\gamma,\bgamma) \coupgrad_{\gamma,\bgamma} \fphi (\eta,\gamma_x^+\eta)\, \inv(\eta) \\
+\frac12\sum_{(\eta,\gamma_x^-)\in\supp}\sum_{(\gamma,\bgamma)\in G^*}c(\eta,\gamma_x^-)\couprate(\gamma_x^-\eta,\eta,\bgamma,\gamma) \coupgrad_{\gamma,\bgamma} \fphi (\eta,\gamma_x^-\eta)\, \inv(\eta)
 \end{align}
 is bounded above by $-\kappa \cE(\phi'(f),f)$. Using reversibility \eqref{eq rev} and the fact that $\fphi$ is symmetric in its arguments, we get that the second summand in \eqref{eq: hardcore G+G-} equals the first one. Combining this with
 \bes
\forall \, \eta,x,y, \quad c(\eta,\gamma_y^-)\leq c(\gamma_x^+\eta,\gamma_y^-),\quad c(\eta,\gamma_y^+)\geq c(\gamma_x^+\eta,\gamma_y^+), \ees we can rewrite  \eqref{eq: hardcore G+G-} as $A+B+C$, where

\begin{align*}
 A & = \sum_{\substack{\eta\in\Omega \\ x,y\in V}} c(\eta,\gamma_x^+)c(\eta,\gamma_y^-) \coupgrad_{\gamma^-_y,\gamma^-_y} \fphi(\eta,\gamma^+_x\eta)\inv(\eta)\\
 & +   \sum_{\substack{\eta\in\Omega \\ x,y\in V}} c(\eta,\gamma_x^+)c(\gamma^+_x\eta,\gamma_y^+) \coupgrad_{\gamma^+_y,\gamma^+_y} \fphi(\eta,\gamma^+_x\eta) \inv(\eta) \\
 B & = \sum_{\substack{\eta,x\sim y \\\eta|_{\bar{N}_y}\equiv 0}}c(\eta,\gamma^+_x)\rho \coupgrad_{\gamma^+_y,\gamma^-_x}\fphi(\eta,\gamma^+_x\eta)\inv(\eta)\\
 C &=  \sum_{\substack{\eta \in \Omega\\x\in V}} c(\eta,\gamma_x^+)\Big[(1-\rho \big|\{y: y\sim x, \eta|_{\bar{N}_y}\equiv 0 \}\big|)\coupgrad_{\zero,\gamma_x^-}\fphi(\eta,\gamma_x^+\eta) +\rho \coupgrad_{\gamma_x^+,\zero}\fphi(\eta,\gamma_x^+\eta)\Big]\inv(\eta).
\end{align*}
Define the set $G^-=\{\gamma_y^-:y\in V \}$ and observe that the first term in $A$ rewrites as 
\bes
\sum_{\substack{\eta \in \Omega\\y\in V}}c(\eta,\gamma)F(\eta,\gamma)\inv(\eta), \quad F(\eta,\gamma)=\mathbf{1}_{\{\gamma\in G^-\}}\sum_{x\in V}c(\eta,\gamma_x^+) \coupgrad_{\gamma,\gamma} \fphi(\eta,\gamma^+_x\eta).
\ees
Therefore, reversibility \eqref{eq rev} yields the equivalent from

\be\label{eq:hard rev}
\sum_{\substack{\eta\in\Omega\\x,y\in V}} c(\eta,\gamma_y^+)c(\gamma_y^+\eta,\gamma_x^+)\coupgrad_{\gamma_y^-,\gamma_y^-}\fphi(\gamma_y^+\eta,\gamma_x^+\gamma_y^+\eta)\inv(\eta).
\ee
Observing that 
\bes c(\eta,\gamma_y^+)c(\gamma_y^+\eta,\gamma_x^+) = c(\eta,\gamma_x^+)c(\gamma_x^+\eta,\gamma_y^+)= \begin{cases} \rho^2, & \quad \mbox{if $x\nsim y$, $\eta\big|_{\bar{N}_x\cup \bar{N}_y}\equiv 0$}\\
0, & \quad\mbox{otherwise,} \end{cases}\ees
and that 
\bes
x\nsim y, \,\eta\big|_{\bar{N}_x\cup \bar{N}_y}\equiv 0 \Rightarrow \gamma_y^+\gamma_x^+\eta=\gamma_x^+\gamma_y^+\eta=\eta+\delta_x+\delta_y,
\ees
we obtain that \eqref{eq:hard rev} equals
\bes
-\sum_{\substack{\eta\in\Omega\\x,y\in V}} c(\eta,\gamma_x^+)c(\gamma_x^+\eta,\gamma_y^+)\coupgrad_{\gamma_y^+,\gamma_y^+}\fphi(\eta,\gamma_x^+\eta)\inv(\eta),
\ees
from which $A=0$ follows. Next, we consider $B$. Recalling that $(\eta,\gamma_x^+)\in \cS$ if and only if $\eta\big|_{\bar{N}_x}\equiv 0$
we can rewrite $B$ in a symmetric way exchanging the labels $x$ and $y$ 
\begin{align*}
2B =  \rho\sum_{\substack{\eta,x\sim y \\\eta|_{\bar{N}_x\cup \bar{N}_y}\equiv0}}[c(\eta,\gamma^+_x)-c(\eta,\gamma^+_{y})] [\fphi(\eta,\gamma_y^+\eta)-\fphi(\eta,\gamma^+_x\eta)]
\end{align*}
This last expression is worth $0$ since $\eta|_{\bar{N}_x\cup \bar{N}_y}\equiv0$ implies $c(\eta,\gamma^+_x)=c(\eta,\gamma^+_{y})=\rho$. Using the fact that $\fphi$ vanishes on the diagonal and the definition of $\Delta$ we get that 
\begin{align*}
C &\leq - (1-\rho(\Delta-1)) \sum_{\substack{\eta \in \Omega\\x\in V}} c(\eta,\gamma_x^+) \Phi(\eta,\gamma_x^+\eta)\inv(\eta)\\
&= -(1-\rho(\Delta-1))\cE(\phi'(f),f),\end{align*}
where to get the last equality we used the reversibility property \eqref{eq rev}. Thus, we have proven that \eqref{eq: discrete gradient estimate} holds with $\kappa'=\kappa$. From the choice of coupling rates we made, it is easily seen that \eqref{def: kappa''} holds with  $\kappa''=\bar\kappa$ and \eqref{def: kappa'''} holds with $\kappa'''=\kappa$. Proposition \ref{cor: sufficient cond} gives the desired conclusion.
\end{proof}

\subsection{Zero-range dynamics}

For zero range dynamics on the complete graph with $L$ vertices, the state space is

\bes \Omega=\left\{\eta:\{1,\ldots,L\}\rightarrow \N \, \text{s.t.}\, \sum_{x=1}^L\eta_x =N \right\},\ees

where $N$ is the number of particles $\eta_x=k$ means that there are $k$ particles at site $x$. As before, for $x\in\{1,\ldots,L\}$ we define $\delta_x\in\Omega$ as 

\bes
(\delta_x)_y = \begin{cases}1, & \quad \mbox{if $y=x$ } \\ 0, & \quad \mbox{otherwise.} \end{cases}
\ees

The set of moves $G$ is $\{\gamma_{xy},x,y\in \{1,\ldots,L\}\}$, the move $\gamma_{xy}$ being defined by

\bes
\gamma_{xy}\eta=\begin{cases}
\eta+\delta_y-\delta_x, \quad \mbox{if $\eta_x>0$}\\
\eta \quad \mbox{otherwise}.
\end{cases}
\ees

The move $\gamma_{xy}$ transfers a particle from site $x$ to site $y$. For any site $x$ we consider a non negative function $c_x:\N\rightarrow \mathbb{R}_{\geq 0}$ such that $c_x(0)=0$. $c_x(k)$ is the rate at which particles are expelled from site $x$ when $k$ particles are on site. Thus, the transition rates for the zero range dynamics rates are given by

\be\label{eq: rates zero range}
\forall\eta\in\Omega,1\leq x,y\leq L,  \quad c(\eta,\gamma_{xy}) = \frac{1}{L} c_x(\eta_x).
\ee

and the generator is
\bes
\cL f(\eta) = \frac{1}{L} \sum_{ x,y \leq L} \, c_x(\eta_x)\nabla_{\gamma_{xy}}f(\eta).
\ees
Assumptions \ref{H0},\ref{H1} are easily seen to be satisfied. In particular, $(\gamma_{xy})^{-1}=\gamma_{yx}$ and the reversible measure is given by (see \cite{caputo2009convex})

\bes
\inv(\eta) = \frac{1}{Z}\prod_{x: \eta_x>0} \prod_{k=1}^{\eta_x}\frac{1}{c_{x}(k)},\quad \eta\in\Omega.
\ees
where $Z$ is the normalization. 

\begin{theorem}\label{thm: zero range}
Assume that there exist $c,\delta\geq0$ such that
\be\label{zero range perturbative cond}
\delta \leq c, \quad  c \leq c_x(k+1)-c_x(k)\leq c+\delta, \quad \forall k\in \N,x\in V .
\ee
Then for the zero range dynamics the following holds
\bei
\item[(i)] If $\phi$ satisfies \ref{H2}, then the convex Sobolev inequality \eqref{convex sob} holds with $\kappa_{\phi}=c-\delta$. In particular, MLSI holds with $\kappa_1=c-\delta$.
\item[(ii)] For $\alpha \in (1,2]$, the Beckner inequality \eqref{beckner ineq} holds with $\kappa_{\alpha}=\alpha c-\delta$.
 \eei
 \end{theorem}

To prove this result we modify slightly the pattern followed so far: instead of using Corollary \ref{cor: sufficient cond} we take one step back and use the identity \eqref{eq: organize} from Lemma \ref{lem: coupling organize} as a starting point.
\begin{remark}
MLSI has been obtained in \cite{caputo2009convex} with the same estimate for $\kappa_1$ as the one given by
Theorem \ref{thm: zero range}. In the recent preprint \cite{hermon2019entropy} MLSI is established removing the assumption that $c\leq \delta$ and only assuming $\delta<+\infty$. This was done using the so called martingale method. As it is remarked in \cite{caputo2009convex}, the entropy may fail to be convex if no relation is imposed between $c$ and $\delta$ and the estimates $\kappa_1\geq c-\delta$ seems to be the best one obtained so far using the Bakry \'Emery approach. Lower bounds for the entropic Ricci curvature of the zero range dynamics have been obtained in \cite{fathi2016entropic}. Concerning Beckner's inequalities with $\alpha\in (1,2)$, the only results we are aware of are those in \cite{jungel2017discrete} on which Theorem \ref{thm: zero range} improves. Indeed the constant obtained there is $\alpha c -(3+2^{\alpha-2}-\alpha)\delta$. Finally, Theorem \ref{thm: zero range} establishes the exponential decay of $\phi$-entropies for of arbitrary functions $\phi$ satisfying \ref{H2}, whereas the results of \cite{jungel2017discrete} cover $\phi=\phi_{\alpha}$. Concerning the spectral gap $(\alpha=2)$ better estimates are known. In particular under the same assumptions of Theorem \ref{thm: zero range} the bound $\kappa_2\geq 2\Delta$ is found in \cite{boudou2006spectral}. A slight modification of the proof of Theorem \ref{thm: zero range} can be used to obtain the same lower bound.
\end{remark}

\begin{proof}
We first prove (i). Let $(\eta,\gamma_{xy})\in\cS$, i.e. $\eta_x>0$. Define

\be\label{def: coup rates zero range}
\couprate(\eta,\gamma_{xy}\eta,\gamma,\bgamma)=\begin{cases}
\frac{1}{L}\min\{c_z(\eta_z),c_z(\gamma_{xy}\eta_z)\}, & \quad \mbox{if $\gamma=\bgamma=\gamma_{zw}$,}\\
\frac{1}{L}[c_x(\eta_x)-c_x(\eta_x-1)-c], & \quad \mbox{if $\gamma=\gamma_{xw},\bgamma=\zero$,}\\
\frac{1}{L}[c_y(\eta_y+1)-c_y(\eta_y)-c], & \quad \mbox{if $\gamma=\zero,\bgamma=\gamma_{yw}$,}\\
\frac{1}{L}c, & \quad \mbox{if $\gamma=\gamma_{xw},\bgamma=\gamma_{yw}$,}\\
0, & \quad \mbox{ otherwise.}
\end{cases}
\ee
It can be verified with a direct calculation that \eqref{def: coup rates int rw} defines admissible coupling rates. In order to do this, we observe that  since $c_z(\cdot)$ is an increasing function

\bes
\min\{c_z(\eta_z),c_z(\gamma_{xy}\eta_z)\} = \begin{cases}
c_z(\eta_z-1), \quad \mbox{if $\eta_z>0$ and $x=z,y\neq z$}\\
c_z(\eta_z), \quad \mbox{otherwise.}
\end{cases}
\ees

Consider now $\phi$ satisfying \ref{H2},$f>0$ and the expression for$\frac{\De}{\De t} 2\cE(\phi'(f_t),f_t)$ provided by \eqref{eq: organize}. Using the convexity of $\Phi$ we get

\begin{align}\label{eq: neutral terms zero range}
\nonumber&\sum_{\substack{(\eta,\gamma_{xy})\in\supp \\ zw}}c(\eta,\gamma_{xy})\couprate(\eta,\gamma_{xy}\eta,\gamma_{zw},\gamma_{zw})\mathrm{D}\Phi(f(\eta),f(\gamma_{xy}\eta)) \cdot \colvec{\nabla_{\gamma_{zw}}f(\eta)}{\nabla_{\gamma_{zw}}f(\gamma_{xy}\eta)}\inv(\eta)\\
&\leq\sum_{\substack{(\eta,\gamma_{xy})\in\supp \\ zw}}c(\eta,\gamma_{xy})\couprate(\eta,\gamma_{xy}\eta,\gamma_{zw},\gamma_{zw})\coupgrad_{\gamma_{zw},\gamma_{zw}}\fphi(\eta,\gamma_{xy}\eta)\inv(\eta):=A
\end{align}

Using again the convexity of $\Phi$  and the fact that $\Phi$ vanishes on the diagonal we obtain that

\begin{align}\label{eq: good terms zero range}
\nonumber\sum_{\substack{(\eta,\gamma_{xy})\in\supp\\ w\leq L}}c(\eta,\gamma_{xy})&\couprate(\eta,\gamma_{xy}\eta,\gamma_{xw},\gamma_{yw})\mathrm{D}\Phi(f(\eta),f(\sigma\eta)) \cdot \colvec{\nabla_{\gamma_{zw}}f(\eta)}{\nabla_{\gamma_{zw}}f(\sigma\eta)}\inv(\eta) \\
&\leq -c \sum_{(\eta,\gamma_{xy})\in\supp}c(\eta,\gamma_{xy})\fphi(\eta,\gamma_{xy}\eta)\inv(\eta)=-2c\cE(\phi'(f),f).
\end{align}

Plugging \eqref{eq: neutral terms zero range} and \eqref{eq: good terms zero range} back into \eqref{eq: organize} we obtain, in view of our choice of coupling rates \eqref{def: coup rates zero range},
\be\label{eq: Dirichlet form derivative zero range}
\frac{\De}{\De t}2\cE(\phi'(f_t),f_t)\Big|_{t=0} \leq -2c\cE(\phi'(f),f)+A+B+C,
\ee
 where 
\begin{align*}
A&= \frac{1}{L^2}\sum_{\eta,xy,zw} c_x(\eta_x) \min\{c_z(\eta_z),c_z(\gamma_{xy}\eta_z) \}\coupgrad_{\gamma_{zw},\gamma_{zw}}\fphi(\eta,\gamma_{xy}\eta)\inv(\eta)\\
B&= \frac{1}{L^2}\sum_{\eta,xy,w}c_x(\eta_x) (c_x(\eta_x)-c_x(\eta_x-1)-c) \partial_a \Phi(f(\eta),f(\gamma_{xy}\eta)) \nabla_{\gamma_{xw}}f(\eta)\inv(\eta)\\
C&= \frac{1}{L^2}\sum_{\eta,xy,w}c_x(\eta_x) (c_y(\eta_y+1)-c_x(\eta_y)-c) \partial_b \Phi(f(\eta),f(\gamma_{xy}\eta)) \nabla_{\gamma_{yw}}f(\gamma_{xy}\eta)\inv(\eta)
\end{align*}
In the rest of the proof, we shall show that $A=0,B=C$ and $2B \leq 2\delta \cE(\phi'(f),f)$. Once this is done, by using these relations in \eqref{eq: Dirichlet form derivative zero range}, we obtain the proof of (i). The proof that $A=0$ is done separately at Lemma \ref{lem: neutral terms zero range}.  The other relations are proven following very closely of Lemma 4.1 in \cite{caputo2009convex} that deals with the case $\phi=\phi_1$. This proof adapts in a straighforward way to the current more general setup and we refer to it whenever the use of the reversibility relation \eqref{eq rev} is not explained in full details in what follows. Exchanging the labels $x$ and $y$ and using that $\inv(\eta)c_y(\eta_y) = \inv(\gamma_{yx}\eta)c_x(\gamma_{yx}\eta_x)$ for all $\eta$ such that $\eta_y>0$ we get that

\begin{align*}
C &= \frac{1}{L^2}\sum_{\eta,xy,w}c_x(\gamma_{yx}\eta_x) (c_x(\gamma_{yx}\eta_x)-c_x(\gamma_{yx}\eta_x-1)-c) \partial_b \Phi(f(\eta),f(\gamma_{yx}\eta))\nabla_{\gamma_{xw}}f(\gamma_{yx}\eta)\inv(\gamma_{yx}\eta)\\
 &= \frac{1}{L^2}\sum_{\eta,xy,w}c_x(\eta_x) (c_x(\eta_x)-c_x(\eta_x-1)-c) \partial_b \Phi(f(\gamma_{xy}\eta),f(\eta)) \nabla_{\gamma_{xw}}f(\eta)\inv(\eta).
\end{align*}
Since $\Phi$ is symmetric in its arguments, we have $\partial_b \Phi(f(\gamma_{xy}\eta),f(\eta)) =\partial_a \Phi(f(\eta),f(\gamma_{xy}\eta))$ and therefore $C=B$. Let us now focus on $B$. Recalling the definition of $\Phi$ and exchanging the labels $y$ and $w$ we get 

\begin{align*}
    -L^2 B &= \sum_{\eta,xy,w}c_x(\eta_x) (c_x(\eta_x)-c_x(\eta_x-1)-c) [\phi''(f(\eta)) \nabla_{\gamma_{xw}}f(\eta) \nabla_{\gamma_{xy}}f(\eta) ]\inv(\eta) \\
    &+\sum_{\eta,xy,w}c_x(\eta_x) (c_x(\eta_x)-c_x(\eta_x-1)-c) [\phi'(f(\gamma_{xw} \eta)) -\phi'(f(\eta))] \nabla_{\gamma_{xy}}f(\eta) \inv(\eta) 
\end{align*}
Using the convexity of $\phi$, we obtain that the first summand in the above expression is non negative. Concerning the second summand, rewriting $\phi'(f(\gamma_{xw}\eta))-\phi'(f(\eta))=\phi'(f(\gamma_{xw}\eta))-\phi'(f(\gamma_{xy}\eta))+\phi'(f(\gamma_{xy}\eta))-\phi'(f(\eta))$ allows to isolate another non-negative term. Therefore,
\begin{align*}
-L^2 B&\geq \sum_{\eta,xy,w}c_x(\eta_x) (c_x(\eta_x)-c_x(\eta_x-1)-c) [\phi'(f(\gamma_{xw} \eta)) -\phi'(f(\gamma_{xy}\eta))]\nabla_{\gamma_{xy}}f(\eta)\inv(\eta)\\
&=\sum_{\eta,xy,w}c_x(\eta_x) (c_x(\eta_x)-c_x(\eta_x-1)-c) [\phi'(f(\gamma_{xw} \eta)) -\phi'(f(\gamma_{xy}\eta))]f(\gamma_{xy}\eta)\inv(\eta):=B.1.
\end{align*}
where to obtain the last inequality we used that $\sum_{y,w}f(\eta) [\phi'(f(\gamma_{xw} \eta)) -\phi'(f(\gamma_{xy}\eta))]=0$. Next, we rewrite $B.1$ twice using reversibility. For the first one we use $\inv(\eta)c_x(\eta_x)=\inv(\gamma_{xy}\eta)c_y(\gamma_{xy}\eta_y)$ to obtain

\begin{align}\label{eq: B1 via rev 1}
\nonumber B.1&=\sum_{\substack{\eta,xy,w \\ x\neq y, \eta_x>0}}c_y(\gamma_{xy}\eta_y) (c_x(\gamma_{xy}\eta_x+1)-c_x(\gamma_{xy}\eta_x)-c) [\phi'(f(\gamma_{yw}\gamma_{xy} \eta)) -\phi'(f(\gamma_{xy}\eta))]f(\gamma_{xy}\eta)\inv(\gamma_{xy}\eta) \\
\nonumber&+\sum_{\substack{\eta,xy,w\\ x=y}}c_y(\eta_y) (c_x(\eta_x)-c_x(\eta_x-1)-c) [\phi'(f(\gamma_{yw} \eta)) -\phi'(f(\eta))]f(\eta)\inv(\eta)\\
\nonumber&=\sum_{\substack{\eta,xy,w\\ x\neq y}}c_y(\eta_y) (c_x(\eta_x+1)-c_x(\eta_x)-c) [\phi'(f(\gamma_{yw} \eta)) -\phi'(f(\eta))]f(\eta)\inv(\eta) \\
&+\sum_{\substack{\eta,xy,w\\ x=y}}c_y(\eta_y) (c_x(\eta_x)-c_x(\eta_x-1)-c) [\phi'(f(\gamma_{yw} \eta)) -\phi'(f(\eta))]f(\eta)\inv(\eta).
\end{align}
If we use $\inv(\eta)c_x(\eta_x)=\inv(\gamma_{xw}\eta)c_w(\gamma_{xw}\eta_w)$, then we get 
\begin{align}\label{eq: B1 via rev 2}
\nonumber B.1&=\sum_{\substack{\eta,xy,w\\x\neq w, \eta_x>0}}c_w(\gamma_{xw}\eta_w) (c_x(\gamma_{xw}\eta_x+1)-c_x(\gamma_{xw}\eta_x)-c) [\phi'(f(\gamma_{xw} \eta)) -\phi'(f(\gamma_{wy}\gamma_{xw}\eta))]f(\gamma_{wy}\gamma_{xw}\eta)\inv(\gamma_{xw}\eta)\\
\nonumber&+\sum_{\substack{\eta,xy,w\\x=w}}c_w(\eta_w) (c_x(\eta_x)-c_x(\eta_x-1)-c) [\phi'(f(\eta)) -\phi'(f(\gamma_{wy}\eta))]f(\gamma_{wy}\eta)\inv(\eta)\\
\nonumber&=\sum_{\substack{\eta,xy,w \\ x\neq w}}c_w(\eta_w) (c_x(\eta_x+1)-c_x(\eta_x)-c) [\phi'(f(\eta)) -\phi'(f(\gamma_{wy}\eta))]f(\gamma_{wy}\eta)\inv(\eta)\\
&+\sum_{\substack{\eta,xy,w\\x=w}}c_w(\eta_w) (c_x(\eta_x)-c_x(\eta_x-1)-c) [\phi'(f(\eta)) -\phi'(f(\gamma_{wy}\eta))]f(\gamma_{wy}\eta)\inv(\eta).
\end{align}
Exchanging the labels $y,w$ in \eqref{eq: B1 via rev 2}, summing the result with \eqref{eq: B1 via rev 1} and using \eqref{zero range perturbative cond} yields
\bes
\frac{1}{L^2}B.1 \geq -\frac{\delta}{2L} \sum_{\eta,y,w}c_y(\eta_y) [\phi'(f(\gamma_{yw} \eta)) -\phi'(f(\eta))][f(\gamma_{yw}\eta)-f(\eta)]\inv(\eta),
\ees
from which it follows that $2B\leq 2\delta\cE(\phi'(f),f)$. The proof of (i) is now complete. The proof of (ii) is almost identical, the only difference being that we can replace \eqref{eq: good terms zero range} with a better bound. Indeed, if instead of the convexity of $\Phi_{\alpha}$ we use \eqref{eq:Beck improve}, then we obtain

\begin{align*}
\nonumber\sum_{\substack{(\eta,\gamma^{xy})\in\supp \\ w\leq L}}c(\eta,\gamma_{xy})\couprate(\eta,\gamma_{xy}\eta,\gamma_{xw},\gamma_{yw})\mathrm{D}\Phi(f(\eta),f(\gamma_{xy}\eta)) \cdot \colvec{\nabla_{\gamma_{xw}}f(\eta)}{\nabla_{\gamma_{yw}}f(\gamma_{xy}\eta)}\inv(\eta) \\
\leq - c\alpha \sum_{(\eta,\gamma_{xy})\in\supp}c(\eta,\gamma_{xy})\fphi(\eta,\gamma_{xy}\eta)\inv(\eta)=-c\alpha\cE(\phi'(f),f).
\end{align*}

\end{proof}

\begin{lemma}\label{lem: neutral terms zero range}
For any $f>0$ and $\phi:\R\rightarrow\R$ we have $A=0$, where

\bes
A=\frac{1}{L^2}\sum_{\eta,xy,zw} c_x(\eta_x) \min\{c_z(\eta_z),c_z(\gamma_{xy}\eta_z) \}\coupgrad_{\gamma_{zw},\gamma_{zw}}\fphi(\eta,\gamma_{xy}\eta)\inv(\eta).
\ees
\end{lemma}

\begin{proof}We begin by observing if $x=y$ or $z=w$, then $\coupgrad_{\gamma_{zw},\gamma_{zw}}\fphi(\eta,\gamma_{xy}\eta)=0$. 
Therefore we can assume w.l.o.g. that $x\neq y$ and $z\neq w$ in the summation that defines $A$. Substracting from $A$ the null term  
\bes \sum_{x\neq y}\sum_{\eta\in\Omega}\cL g^{xy}(\eta)\inv(\eta), \text{where} \quad g^{xy}(\eta) = c_x(\eta_x)\fphi(\eta,\gamma_{xy}\eta)\ees 
and using that for $x\neq y,z\neq w$ we have
\beas
c_x(\eta_x) \min\{ c_z(\eta_z),c_z(\gamma_{xy}\eta_z)\} - c_z(\eta_z)c_x(\gamma_{zw}\eta_x)=\begin{cases} c_z(\eta_z)[c_x(\eta_x)-c_x(\eta_x+1)],&\quad \mbox{if $x=w$, $\eta_z>0$}\\ 0, &\quad \mbox{otherwise,} \end{cases}\\
c_x(\eta_x) \min\{ c_z(\eta_z),c_z(\gamma_{xy}\eta_z)\} - c_z(\eta_z)c_x(\eta_x)=\begin{cases} c_x(\eta_x)[c_x(\eta_x-1)-c_x(\eta_x)],&\quad \mbox{if $x=z,\eta_x>0$}\\ 0, &\quad \mbox{otherwise.} \end{cases}
\eeas
we obtain that $L^2A=A.1-A.2$ with
\begin{align*}
A.1 &= \sum_{\substack{\eta:\eta_z>0 \\ x\neq y, z\neq x}} c_z(\eta_z)[c_x(\eta_x)-c_x(\eta_x+1)]\fphi(\gamma_{zx}\eta,\gamma_{zy}\eta)\inv(\eta),\\
A.2 &= \sum_{\substack{\eta:\eta_x>0 \\ x\neq y, w\neq x}} c_x(\eta_x) [c_x(\eta_x-1)-c_x(\eta_x)]\fphi(\eta,\gamma_{xy}\eta)\inv(\eta).
\end{align*}
By reversibility we have that if $\eta_z>0$ then $\inv(\eta)c_z(\eta_z)=\inv(\gamma_{zx}\eta)c_x(\gamma_{zx}\eta)$. Therefore
\begin{align*}
A.1&= \sum_{\substack{x\neq y\\ z\neq x}}\sum_{\eta:\eta_z>0 } c_x(\gamma_{zx}\eta_x)[c_x(\gamma_{zx}\eta_x-1)-c_x(\gamma_{zx}\eta_x)]\fphi(\gamma_{zx}\eta,\gamma_{xy}\gamma_{zx}\eta)\inv(\gamma_{zx}\eta).
\end{align*}
Since the map $\gamma_{zx}$ induces a bijection between the sets $\{\eta:\eta_z>0\}$ and $\{\eta:\eta_x>0\}$ we get that $A.1=A.2$, and finally that $A=0$. 
\end{proof}

\section{Appendix}

\paragraph{Proof of Lemma \ref{lem: reny admissible}}
\begin{proof}
We begin by showing the convexity of $\Phi_{\alpha}$. The case $\alpha=1$ is well known. We do the proof for the sake of completeness. In this case $\Phi_1(a,b)=(b-a)(\log b -\log a)$ and for all $a,b>0$ we get

\bes
\mathrm{Hess}\Phi_{\alpha}(a,b)  \begin{pmatrix} \frac1a+\frac{b}{a^2} &-(\frac1b+\frac1a) \\-(\frac1b+\frac1a) & 
\frac1b+\frac{a}{b^2} \end{pmatrix},
\ees
which is a positive semidefinite matrix since the trace is positive and the determinant is $0$.

For $\alpha\in(1,2)$ we get from the definition of $\Phi_{\alpha}$ that $\Phi_{\alpha}(a,b)=\frac{\alpha}{\alpha-1}(a-b)(a^{\alpha-1}-b^{\alpha-1})$. Therefore for all $a,b>0$:
\bes
\mathrm{Hess}\Phi_{\alpha}(a,b) = \begin{pmatrix} \alpha^2 a^{\alpha-2}+\alpha(2-\alpha)a^{\alpha-3}b& -\alpha (a^{\alpha-2} +b^{\alpha-2})\\ -\alpha (a^{\alpha-2} +b^{\alpha-2}) & \alpha^2 b^{\alpha-2}+\alpha(2-\alpha)b^{\alpha-3}a \end{pmatrix}
\ees
It is easily seen that the trace of this matrix is always nonnegative. To conclude, we verify that so is its determinant which, after some calculations turns out to be

\bes
\alpha^2 (-(a^{\alpha-2}-b^{\alpha-2})^2+ \alpha(2-\alpha) (ab)^{\alpha-3}(a-b)^2 )
\ees

To show that the above expression is always non negative we assume w.l.o.g. that $b\geq a$,divide by the positive constant $\alpha^{2}a^{2\alpha-4}$ and set $z=a/b$. We obtain the desired conclusion if we can show that 
\bes
\inf_{z\geq 1} -(z^{\alpha-2}-1)^2+(2-\alpha)\alpha z^{\alpha-3}(z-1)^2\geq 0
\ees
Set $g(z)= -(z^{\alpha-2}-1)^2+(2-\alpha)\alpha z^{\alpha-3}(z-1)^2$. A standard calculation yields 
\bes
g'(z)=2(2-\alpha)(z^{\alpha-2}-1)z^{\alpha-3} +(2-\alpha)\alpha(\alpha-3)z^{\alpha-4}(z-1)^2+2(2-\alpha)\alpha
z^{\alpha-3}(z-1)\ees
Since $(2-\alpha-2)z^{\alpha-4}\geq0$ we have $g'(z)\geq 0$ iff
\beas
2(z^{\alpha-2}-1)z+\alpha(\alpha-3)(z-1)^2 +2\alpha(z-1)z\geq 0\\
\Leftrightarrow 2(z^{\alpha-2}-1)z + (z-1)[\alpha(\alpha-1)z+\alpha(3-\alpha)]\geq 0
\eeas

Now observe that, uniformly on $z\geq 1$ we have $z^{\alpha-2}\geq z^{-1}$. Therefore 
\bes
 2(z^{\alpha-2}-1)z + (z-1)[\alpha(\alpha-1)z+\alpha(3-\alpha)] \geq (z-1)[\alpha(\alpha-1)z+\alpha(3-\alpha)-2]\geq0
\ees
since $\alpha(3-\alpha)-2 \geq 0$ in $[1,2]$, from which we get that $g(z)$ is increasing on $[1,+\infty)$. Since $\lim_{z\rightarrow 1} g(z)=0$, the desired conclusion follows. Let us now turn to the proof of \eqref{eq:Beck improve}. Let $\alpha\in(1,2]$. If $a'=b'$ we obtain after some calculations that
\begin{align*}\label{eq: Beck convexity}
\nonumber &\left(\Phi_{\alpha}(a',b')-\Phi_{\alpha}(a,b)  - \mathrm{D}\Phi_{\alpha} (a,b) \cdot \colvec{a'-a}{b'-b} \right) \\
&=\alpha a^{\alpha-2}(b-a)(a'-a)+\alpha b^{\alpha-2}(a-b)(b'-b)\\
&=\alpha a'(a^{\alpha-2}(b-a)+  b^{\alpha-2}(a-b) ) + (\alpha-1)\Phi_{\alpha}(a,b).
\end{align*} 
Using the concavity of $x\mapsto x^{\alpha-1}$ to bound $a^{\alpha-2}(b-a)$ and $b^{\alpha-2}(a-b)$ yields \eqref{eq:Beck improve}. It remains to prove \eqref{eq: MLSI improve}. A direct calculation allows gives that the right hand side of \eqref{eq: MLSI improve}  is worth $(a-b)^2 (1/a+1/b)$. Therefore \eqref{eq: MLSI improve} is equivalent to ask that for all $a,b>0$
\bes
(1/a+1/b) \geq 2\frac{(\log b-\log a)}{b-a}
\ees
Assume w.l.o.g. that $a <b$. We have $2\frac{(\log b-\log a)}{b-a} =\frac{2}{(b-a)}\int_a^b1/s \De s$. Using the convexity of $s\mapsto 1/s$ we have

\bes
\frac{2}{(b-a)}\int_a^b\frac1s \De s\leq\frac{2}{(b-a)^2}\int_a^b\frac{(s-a)}{a}+\frac{(b-s)}{b} \,\De s =\frac{1}{a}+\frac{1}{b}
 \ees
\end{proof}

\paragraph{Proof of Lemma \ref{lem: 1 jump approx}}

\begin{proof}
\underline{Step 1: Localization via stopping times}
In the proof, $C$ denotes a generic positive constant whose value may change from one expression to another. Let $(X^{\eta}_t)_{t\geq 0}$ be a continuous time random walk with generator $\cL$ and initial distribution $\delta_{\eta}$. We denote $T_1$ be the first jump time of the walk and define $\tilde{\mu}_t$ as the law of $X^{\eta}_{t\wedge T_1}$.  A straightforward calculation gives that 

\bes
\tilde\mu_{t} = \exp(-C(\eta)t)\delta_{\eta}+ \sum_{\gamma\in G} \frac{1-\exp(-C(\eta)t)}{C(\eta)}c(\eta,\gamma) \delta_{\gamma\eta}.
\ees
where $C(\eta)=\sum_{\gamma\in G}c(\eta,\gamma) $. It is easily seen that that $W^p_p(\bar\mu_t,\tilde{\mu}_t) \leq C t^{2}$. Therefore it suffices to show \eqref{eq: 1 jump approx} replacing $\bar{\mu}_t$ with $\tilde{\mu}_t$. Indeed, in this case we would have
\bes
W^p_p(\mu_t,\bar\mu_t) \leq (W_p(\mu_t,\tilde\mu_t)+W_p(\tilde\mu_t,\bar\mu_t))^p \leq 2^{p-1}(W^p_p(\mu_t,\tilde\mu_t)+W^p_p(\tilde\mu_t,\bar\mu_t))\leq Ct^2.
\ees
To this aim, we can exploit the fact that 
\be\label{eq: localization bound 1}
W^p_p(\mu_t,\tilde{\mu_t}) \leq \bbE\big[d(X^{\eta}_t,X^{\eta}_{t\wedge T_1})^p \big].
\ee
$(X^{\eta}_t)_{\geq 0}$ waits an exponential random time $T_1$ and then jumps to the state $\gamma\eta$ with a probability proportional to $c(\eta,\gamma)$. Therefore, using the strong Markov property we obtain
\be\label{eq: localization bound 2}
\bbP[d(X^{\eta}_t,X^{\eta}_{t\wedge T_1})\geq k] = \sum_{\gamma\in G} c(\eta,\gamma) \int_{0}^t\exp(-C(\eta)s) \bbP[d(X^{\gamma\eta}_{t-s},\gamma\eta)\geq k ]\De s
\ee
where $(X^{\gamma\eta}_{r})_{r\geq0}$ is a continuous time Markov chain with generator $\cL$ started at $\gamma\eta$. \\
\underline{Step 2: Bound on $\bbP[d(X^{\gamma\eta}_{t-s},\gamma\eta)\geq k ]$}. Fix $r>0$ and observe that 
\be\label{eq: tail partition}
\{d(X^{\gamma\eta}_r,\gamma\eta)\geq k\} =  \bigcup_{\substack{I^{+}\subseteq \{1,...,d\} \\ k_1+\ldots+k_d=k}} \{ (X^{\gamma\eta}_r)_i-\gamma\eta_i \geq k_i, i\in I^+\} \cap \{ \gamma\eta_i-(X^{\gamma\eta}_r)_i \geq k_i, i\in I^-\} .
\ee
For $i\leq d$ consider the counting process $(N^i_{r})_{r\geq 0}$ defined by
\bes
N^i_r-N^i_{r^-} = \begin{cases} 1, & \quad \mbox{if $i\in I^{+}, X^{\gamma\eta}_{r}=\gamma_i^+X^{\gamma\eta}_{r^-}$}\\
1, & \quad \mbox{if $i\in I^{-}, X_{r}=\gamma_i^-X^{\gamma\eta}_{r^-}$ and $(X^{\gamma\eta}_{r^-})_{i}\leq \gamma\eta_i$ }\\
0, & \mbox{otherwise}
\end{cases}
\ees
We remark that, because of \eqref{hyp: non explosion} we have $A<+\infty$, where 
\bes
A:=\max\{\sup\{ c(\eta,\gamma_i^+):\eta\in \N^d\}, \, \sup\{ c(\eta,\gamma_i^-):\eta_i \leq \eta_i+1 \} \}
\ees
Moreover, observe that for any $i\in I^+$ and conditionally on $X^{\gamma\eta}_{[0,r^{-}]}$, the process $N^i_{\cdot}$ jumps up at rate $c(X^{\gamma\eta}_{r^-},\gamma^+_i)\leq A$. If $i\in I^{-}$, then $N^i_{\cdot}$ jumps up at rate $c(\eta,\gamma_i^-)\leq A$ when $X^{\gamma\eta}_{r^-} \leq \gamma\eta_i$, and at rate $0$ when $X^{\gamma\eta}_{r^-} > \gamma\eta_i$. Therefore, for any fixed $r>0$ the random vector $(N^1_r,\ldots,N^r_d)$ is stochastically dominated by a random vector $(M^1_r,\ldots,M^d_r)$ whose components are independent Poisson random variables of parameter $rA$. Using this, we obtain

\begin{align*}
\bbP[\{ (X^{\gamma\eta}_r)_i-\gamma\eta_i \geq k_i, i\in I^+\} &\cap \{ \gamma\eta_i-(X^{\gamma\eta}_r)_i \geq k_i, i\in I^-\} ]\\
&\leq \bbP[N^{1}_r\geq k_1,\ldots,N^{d}_r\geq k_d ] \\
&\leq  \prod_{i=1}^d\bbP[M^{i}_r\geq k_i] \\
&\leq  (\exp(-Ar)+1)^d \prod_{i=1}^d \frac{(Ar)^{k^i}}{k_i!}.
\end{align*}
Using \eqref{eq: tail partition} and summing over all possible choices of $k_1,\ldots,k_d$ and  $I^+\subseteq\{1,...,d\}$ we obtain that for all $r>0$ and $k>0$

\be\label{eq: tail bound}
 \bbP[\{d(X^{\gamma\eta}_r,\gamma\eta)\geq k\}]\leq 2^d(1+\exp(Ar))^d\exp(-Ar d) \frac{(Adr)^k}{k!}.
\ee

\underline{Step 3: conclusion}
 Combining \eqref{eq: tail bound} and \eqref{eq: localization bound 2} we obtain that there exists $C>0$ such that for all $r>0,k>0$
\bes
\bbP[d(X^{\eta}_t,X^{\eta}_{T_1\wedge t})=k]\leq C^k \frac{t^{k+1}}{(k+1)!}
\ees
holds. Using this expression to bound side the right  of \eqref{eq: localization bound 2} we obtain that
\bes
W^p_p(\tilde\mu_t,\mu_t) \leq \sum_{k\geq 1} k^p C^k \frac{t^{k+1}}{(k+1)!} \leq t^2\sum_{k\geq 1} k^p C^k \frac{t^{k-1}}{k!}\leq t^2C.
\ees
\end{proof}

  \paragraph{Acknowledgements}{The author wishes to thank Paolo Dai Pra and Matthias Erbar for useful discussions.}
\bibliographystyle{plain}
\bibliography{Ref}

\end{document}